\documentclass[a4paper,10pt]{article}
\usepackage[french, english]{babel}
\usepackage{xpatch}
\usepackage{csquotes}
\usepackage{geometry}
\usepackage{amsmath,amssymb,amsthm}
\usepackage{IEEEtrantools}
\usepackage{mathtools}
\usepackage{authblk}
\usepackage{tikz}
\usetikzlibrary{calc, positioning, intersections, patterns, arrows.meta,
decorations.markings, matrix}
\usepackage[titletoc,title]{appendix}
\usepackage{dsfont}
\usepackage{microtype}
\usepackage[font=footnotesize, margin=1em]{caption}
\usepackage{hyperref}
\usepackage{cleveref}

\usepackage{stix2}

\definecolor{sienne}{RGB}{136, 45, 23}
\hypersetup{
colorlinks = true,
linkcolor = sienne
}
\allowdisplaybreaks[1]

\newtheorem{theorem}{Theorem}
\newtheorem{prop}[theorem]{Proposition}
\newtheorem{corollary}[theorem]{Corollary}
\newtheorem{lemma}[theorem]{Lemma}
\theoremstyle{definition}

\theoremstyle{remark}
\newtheorem{remark}[theorem]{Remark}

\makeatletter
\patchcmd{\@IEEEeqnarray}{\relax}{\relax\intertext@}{}{}
\makeatother

\makeatletter
\newcounter{proofstep}[theorem]
\newcommand{\step}[1]{%
\refstepcounter{proofstep}%
\vskip-\lastskip\medskip\noindent\textit{Step \arabic{proofstep}: #1. --- }}
\xapptocmd\proof{\setcounter{proofstep}{0}}{}{}
\makeatother

\DeclareMathOperator{\Sp}{Sp}

\DeclareMathOperator{\Span}{Span}

\DeclareMathOperator{\range}{Range}

\DeclareMathOperator{\codim}{codim}

\DeclareMathOperator{\rank}{rank}

\DeclareMathOperator{\Adj}{Adj}
\DeclareMathOperator{\order}{Order}

\newcommand{\bigO}{O}

\newcommand{\diff}[1][-3]{\ensuremath{\mathop{}\mkern#1mu\mathrm{d}}}
\newcommand{\set}{\mathbb}
\newcommand{\N}{{\set{N}}}
\newcommand{\Z}{{\set{Z}}}
\newcommand{\R}{{\set{R}}}
\newcommand{\C}{{\set{C}}}
\newcommand{\T}{{\set{T}}}
\newcommand{\A}{{\mathcal A}}

\newcommand{\coloneqq}{\mathrel{\mathord{:}\mathord=}}
\newcommand{\eq}{\Leftrightarrow}

\newcommand{\h}{{\mathrm h}}
\newcommand{\p}{{\mathrm p}}
\newcommand{\m}{{\mathrm m}}
\newcommand{\eu}{\mathrm e}
\newcommand{\iu}{\mathrm i}

\renewcommand{\epsilon}{\varepsilon}

\newcommand{\kal}[3][]{ [#2 |#3]\ifblank{#1}{}{_{#1}}}

\title{Null-controllability of underactuated linear parabolic-transport systems with constant coefficients}
\author{Armand Koenig\thanks{IMT, Université de Toulouse, CNRS, Université Toulouse III-Paul Sabatier (UPS), Toulouse, France (\texttt{armand.koenig@math.univ-toulouse.fr})}, Pierre Lissy\thanks{CEREMADE,  Universit\'e Paris-Dauphine \& CNRS UMR 7534, Universit\'e PSL, 75016 Paris, France (\texttt{lissy@ceremade.dauphine.fr}).}}
\begin{document}
\maketitle

\begin{abstract}
 The goal of the present article is to study controllability properties of mixed systems of linear parabolic-transport equations, with possibly non-diagonalizable diffusion matrix, on the one-dimensional torus. The equations are coupled by zero or first order coupling terms, with constant coupling matrices, without any structure assumptions on them. The distributed control acts through a constant matrix operator on the system, so that there might be notably less controls than equations, encompassing the case of indirect and simultaneous controllability. More precisely, we prove that in small time, such kind of systems are never controllable in appropriate Sobolev spaces, whereas in large time, null-controllability holds, for sufficiently regular initial data, if and and only if a spectral Kalman rank condition is verified. We also prove that initial data that are not regular enough are not controllable. Positive results are obtained by using the so-called fictitious control method together with an algebraic solvability argument, whereas the negative results are obtained by using an appropriate WKB construction of approximate solutions for the adjoint system associated to the control problem. As an application to our general results, we also investigate into details the case of $2\times2$ systems (\textit{i.e.}, one pure transport equation and one parabolic equation).
\end{abstract}

\vspace*{0.2cm}


\vspace*{0.2cm}


\section{Introduction}
\subsection{Context and state of the art}
Controllability properties of coupled systems of PDEs has attracted a lot of attention this last two decades, due to their link with real-life models and also the specific mathematical difficulties arising in this context. An important part of the literature is devoted to systems where all components of the equations have the same qualitative behaviour (meaning that they are for instance all parabolic, or all hyperbolic, etc.). However, the case where different dynamics are mixed  has been less studied, despite its mathematical interest. Indeed, in this context, the controllability properties of each equation taken separately might be totally different (for instance, the heat equation with distributed control is controllable in arbitrary small time from any open subset \cite{LR95,FI96}, whereas the wave equation with distributed control is controllable in large time and under some geometric conditions \cite{BLR92}), so that the controllability properties of the final coupled system might be difficult to guess. Moreover, when we are considering underactuated systems (in the sense that there are less controls than equations) as in the present article, additional mathematical difficulties are appearing, due notably to the algebraic and analytic effects of the coupling terms, that become predominant in the understanding of the controllability or observability properties of the system under study. Here, in the present article, we aim to study the indirect controllability properties of a model of coupled parabolic-transport equations as introduced in \cite{BKL20}.

Let us mention that many realistic models already studied in the literature can be reformulated in terms of coupled parabolic-transport equations, notably the wave equation with structural damping \cite{RR07,MRR13,CRZ14,GR21}, the heat equation with memory \cite{IP09,GI13}, the 1D-Linearized compressible Navier-Stokes equations \cite{EGGP12,CMRR14,CM15,BCDK22}, or the Benjamin-Bona-Mahony equation \cite{RZ13}. For more details, we also refer to \cite[\S1.4]{BKL20}.  This justifies the interest of studying a general version of coupled parabolic-transport systems as in the present article, that can be seen as an attempt to find a unified framework in order to encompass many existing results of the literature and to generalize them. Other results of interest, related to the present work, are \cite{AMM22}, where the authors study a one-dimensional system of one transport equation and one parabolic equation, for which they prove a non-controllability result in small time by a WKB approach, and \cite{CDM22}, where the authors prove a controllability result in large time for a one-dimensional system of one transport equation and one elliptic equation.

\subsection{Presentation of the parabolic-transport system under study}
Let $T >0$ some final time , $\T = \R/(2\pi \Z)$ the one-dimensional torus, $\omega$ an nonempty open subset of $\T$, $d \in \N^{*}$  (which represents the number of equations in our system) , $m\in\{1,\dots,d\}$ (which represents the number of controls in our system), $A, B, K \in \mathcal M_d(\mathbb R)$ (that are some constant coupling matrices), and  $M\in\mathcal M_{d,m}(\mathbb R)$ (that is a constant control operator).
Our goal is to study the controllability properties of the following coupled system of parabolic-transport equations:
\begin{equation}
\label{Syst}\tag{Sys}
\left\{
\begin{array}{l l}
\partial_t f-  B \partial_{x}^2 f + A \partial_x f + K f  = M u \mathds1_{\omega} &\text{in}\ (0,T)\times \T,\\
f(0,\cdot)=f_0& \text{in}\ \T.
\end{array}
\right.
\end{equation}
Here,  the state is $f\colon [0,T]\times\T \to \R^d$, and  the control is $u\colon [0,T]\times\T \to \R^m$. The exact  regularity chosen for $f$ and $u$ will be made more precise later on.

We assume that
\begin{gather}
d=d_\h + d_\p\ \text{with}\ 1 \leq d_\h < d,\ 1 \leq d_\p < d,\label{dd1d2}\tag{H.1}\\
B = \begin{pmatrix}0&0\\0&D\end{pmatrix},\text{ with } D\in \mathcal M_{d_\p}(\mathbb R),\label{h:B}\tag{H.2}\\
\Re(\Sp(D)) \subset (0, +\infty). \label{h:D}\tag{H.3}
\end{gather}
$d_\h$ represents the number of purely hyperbolic equations, whereas $d_\p$ represents the number of parabolic equations.

Notice that \eqref{h:D} is necessary to ensure that the matrix operator $\partial_t - D\Delta$ is parabolic is the sense of Petrovskii (\cite[Chapter 7, Definition 2]{LSU68}).
Introducing the similar block decomposition for the $d\times d$ matrix $A= \big (\begin{smallmatrix}A'& A_{12}\\A_{21}&A_{22}\end{smallmatrix}\big)$, we make the following hypothesis on the matrix $A'\in \mathcal M_{d_h}(\mathbb R)$
\begin{equation}
\tag{H.4}
A'\text{ is diagonalizable with }\Sp(A')\subset \set R.\label{h:A1}
\end{equation}
Notice that it is well-known that \eqref{h:A1} is necessary (and sufficient, see \cite[\S2.2]{BKL20}) to ensure the well-posedness of \eqref{Syst}.
\subsection{Main results}
To state our results, we need to introduce the following notations:
\begin{equation}\label{eq-def-l}\ell(\omega) \coloneqq \sup\{ |I|;\ I \text{ connected component of }  \T \setminus \omega \},\end{equation}
\[\mu_* \coloneqq \min\{ |\mu|;\ \mu \in \Sp(A') \},\]
and
\begin{equation}\label{eq-def-T*}
T^\ast = T^*(\omega) \coloneqq 
\left\lbrace \begin{array}{ll}
\frac{\ell(\omega)}{\mu_*}  & \text{ if } \mu_* >0, \\
+\infty & \text{ if } \mu_*=0.
\end{array}\right.
\end{equation}
For $n\in \Z$, we also set 
\begin{equation}
    \label{eq-def-Bn}
    B_n \coloneqq -n^2 B -\iu n A -K
\end{equation}
and
\begin{equation}
    \label{eq-def-kalman}
    \kal{B_n}{M} \coloneqq \begin{pmatrix}M&B_nM&\dots & B_n^{d-1}M\end{pmatrix}.
\end{equation}
Our main result is the following one.
\begin{theorem}\label{th-main}
Assume that the hypotheses \eqref{dd1d2}--\eqref{h:A1} hold, that $T>T_*$.

Then,  the spectral Kalman rank condition $\rank(\kal{B_n}{M}) = d$ holds for all $n\in \Z$ if and only if for every $f_0\in H^{4d(d-1)}(\T)^d$, there exists a control $u\in L^2([0,T]\times \omega)^m$ such that the solution $f$ of the parabolic-transport system~\eqref{Syst} with initial condition $f_0$ satisfies $f(T,\cdot) = 0$.
\end{theorem}
\begin{remark}\label{rk-main}
\begin{itemize}
\item Recall that the Kalman rank condition is necessary for the control of ODE systems~\cite[Theorem~1.16]{Coron07}. Therefore,   writing the parabolic-transport system in Fourier, we immediately find that for every $T>0$, the spectral Kalman-rank condition $\forall n\in\Z,\ \rank(\kal{B_n}{M}) = d$ is necessary for the null-controllability of every $H^k$ initial conditions in time $T$.
\item Actually, we prove two slightly stronger versions of this theorem, namely \cref{th-main-p,th-main-k0}, that are useful in order to obtain some controllability results under some constraints on Fourier coefficients of the hyperbolic part of the initial condition (see \cref{th-2x2-h}, \cref{th-2x2-p}, \cref{th-2x2-sim}). 
\item One can refine a little bit the regularity stated in \cref{th-main}, as follows. 
Assume that  $T>T_*$ and that for all $n\in \Z$, the spectral Kalman rank condition $\rank(\kal{B_n}{M}) = d$ holds. Then:
\begin{enumerate}
    \item for every $f_0\in H^{4d(d-1)}(\T)^{d_h}\times H^{4d(d-1)-1}(\T)^{d_p}$, there exists a control $u\in L^2([0,T]\times \omega)^m$ such that the solution $f$ of the parabolic-transport system~\eqref{Syst} with initial condition $f_0$ satisfies $f(T,\cdot) = 0$.
    \item if $A_{12}=0$, for every $f_0\in H^{4d(d-1)}(\T)^{d_h}\times H^{4d(d-1)-2}(\T)^{d_p}$, there exists a control $u\in L^2([0,T]\times \omega)^m$ such that the solution $f$ of the parabolic-transport system~\eqref{Syst} with initial condition $f_0$ satisfies $f(T,\cdot) = 0$.
\end{enumerate}
Indeed, by letting evolve the system freely on a short interval of time, we can show using the method of~\cref{th-regularity} that the parabolic component becomes $H^{4d(d-1)}(\T)^{d_p}$, so that \cref{th-main} can be applied, taking into account that the condition $T>T^*$ is open and that the system is time-invariant.
\item The spectral Kalman rank condition $\rank(\kal{B_n}{M}) = d$ was first introduced in \cite{ABDG09} for coupled systems of heat equations with diagonalizable diffusions (see also \cite{LZ19} for non-diagonalizable diffusions).
\end{itemize}
 \end{remark}
 %We also prove the following necessary condition for the null-controllability to hold.
\begin{theorem}\label{th-necessary-rough}
 Let $\mu\in \Sp(A')$, $N\in \N$ and $T>0$. Assume that every initial condition $f_0\in L^2(\T)^d\cap \{\sum_{|n|>N} X_n \eu^{\iu nx}\}$ is steerable to $0$ in time $T$ with control in $L^2((0,T)\times \omega)$. Then, there exists $V_0\in \ker(A'^*+\mu)$ such that $M^*\big(\begin{smallmatrix} V_0\\0\end{smallmatrix}\big) \neq 0$.
\end{theorem}
\begin{remark}
\Cref{th-main,th-main-p,th-main-k0} only ensures null-controllability of smooth enough initial conditions. \Cref{th-necessary-rough} proves that such a regularity condition is needed in general: even if the time is large enough and if the Kalman rank condition is satisfied for every $n$, it might happen that some $L^2$ initial condition cannot be steered to $0$ with a $L^2$ control.
\end{remark}
\subsection{Precise scope and organization of the article}
This article can be seen as a continuation of \cite{BKL20}, insofar as we generalize the results of the above-mentioned article, since we are able to treat any matrices $A,B,K,M$ without any restrictions on their structure. Indeed, in \cite{BKL20}, the authors treated the case where $M=I_d$ (where no Kalman rank condition is needed), or particular cases where only the parabolic or the hyperbolic parts are controlled, under strong restrictions on the structure of the coupling matrices $A,B$ and $K$ and also on the diffusion matrix $B$.

Let us mention that our results are sharp in terms of the controllability conditions we obtain. However, it is very likely that the initial state space (whose choice is determined by technical reasons coming from the specific strategy we use, that is consuming in terms of regularity, see Section \ref{ssec-alg})  is almost never sharp and depends strongly on the structure of the coupling terms. Finding the exact ``good'' state space remains an open problem that seems to be difficult to solve in all generality.

The article is organized as follows. In \cref{sec-notations}, we give some notations and we gather some existing results that will be used in our proof. \Cref{sec-cs} is devoted to proving that the condition $\rank(\kal{B_n}{M}) = d$ is sufficient in order to obtain our desired controllability result in large time %(the necessity being trivial, see \cref{ssec-kcn}).
The argument is based on a fictitious control argument detailed in \cref{ssec-an}, where we first prove an auxiliary controllability result, in the case $M=I_d$, with regular enough controls for regular enough initial data. Then, in \cref{ssec-alg}, we explain how to obtain a control in the range on $M$ by performing algebraic manipulations.
Notice that the method of fictitious control plus algebraic solvability, that has been introduced in \cite{CL14} in the context of the controllability of PDEs, has been successfully used for various problems \cite{ACO17,DL16,DL18,LL17,CG17,SGM19,SGM19a,DL22}. One of the main novelties here is that the algebraic solvability is not directly performed on the system (or its adjoint as in \cite{DL22}) but on a projected version of the system on its Fourier components. \Cref{sec-cn} is devoted to proving some necessary conditions of controllability. %In \cref{ssec-kcn}, we state that the Kalman rank condition $\rank(\kal{B_n}{M}) = d$ is necessary for the controllability of \eqref{Syst}.
\Cref{sec-wkb} is devoted to constructing WKB solutions. These solutions are used to disprove controllability in small time in \cref{ssec-st} and to prove \cref{th-necessary-rough} in \cref{ssec-pr}. \Cref{sec-2t2} aims to give an application of our results to the particular case of $2\times 2$ systems together with some considerations about the sharpness of our regularity assumptions in this precise setting. To conclude, \cref{app-comp-uniq} proves a general result about a ``control up to a finite-dimensional space plus unique continuation'' strategy that is used in \cref{ssec-an}, in the spirit of \cite{LZ98,BKL20}.

\textbf{Acknowledgement} Armand Koenig is supported by the ANR LabEx CIMI (under grant ANR-11-LABX-0040) within the French State Programme ``Investissements d'Avenir''.

Pierre Lissy is supported by the Agence Nationale de la Recherche, Project TRECOS, under grant ANR-20-CE40-0009.

\section{Some notations and preliminary results}\label{sec-notations}
We will rely on some basic results on the parabolic-transport system~\eqref{Syst} that are already known, see \cite{BKL20}. For the reader convenience, we collect here the notations and results we will use most often, and we will recall some others along the way as they are used.

Let $\mathcal L$ be the unbounded operator on $L^2(\set T)^d$ with domain $H^1(\set T)^{d_\h} \times H^2(\set T)^{d_\p}$ defined by
\[
\mathcal L f = -B\partial_x^2 f + A\partial_xf + Kf.
\]
The operator $-\mathcal L$ generates a strongly continuous semigroup of bounded operators of $L^2(\set T)^d$~\cite[Proposition 11]{BKL20}. Every $H^k(\set T)^d$ is stable by $\eu^{-t\mathcal L}$, and the restriction of $\eu^{-t\mathcal L}$ on $H^k(\set T)^d$ is a strongly continuous semigroup of bounded operators~\cite[Remark 13]{BKL20}. We denote by $S(T,\,f_0,\,u)$ the solution at time $T$ of the parabolic-transport system~\eqref{Syst} with control matrix $M=I_d$  (the identity matrix of size $d$, \textit{i.e.}, we control every component with a different control), initial condition $f_0$ and control $u$.

Let $n_0 \in \set N$ to be chosen large enough later on. We denote by $e_n : x\in \set T \mapsto \eu^{\iu nx}$. We also denote by $E \colon \mathbb C \rightarrow \mathcal M_d(\mathbb C)$ the following function:
\[E(z)=B+zA-z^2K.\]

Let $r>0$ small enough. For $|z|<r$, let $P^\h(z)$ be the eigenprojection on the sum of eigenspaces of $E(z)$ associated to the set of eigenvalues $\lambda(z)\in \Sp(E(z))$ such that $|\lambda(z)|<r$. According to~\cite[Proposition~5]{BKL20}, $P^\h(z)$ satisfies:
\begin{itemize}
    \item $P^\h(0) = \big(\begin{smallmatrix} I&0\\0&0\end{smallmatrix}\big)$;
    \item $z\mapsto P^\h(z)$ is holomorphic;
    \item $P^\h(z)$ is a projection that commutes with $E(z)$;
    \item $P^\h(z) E(z) = \bigO(z)$ as $z\rightarrow 0$.
\end{itemize}
We also set $P^\p(z) = I - P^\h(z)$. This projection $P^\p(z)$ satisfies similar properties as $P^\h(z)$~(\cite[Propositions~6]{BKL20}).

Following \cite[Proposition 18]{BKL20}, we denote by $F^0$ the space of frequencies less than $n_0$ and by $F^\h$ (respectively $F^\p$) the space of hyperbolic frequencies greater than $n_0$ (respectively the space of parabolic frequencies greater than $n_0$), \textit{i.e.}
\[\begin{array}{r@{}l@{}l@{}l@{}l}
F^0       &{}={}&\bigoplus_{|n|\leq n_0}\Span(e_n); & \\
F^\p       &{}={}&\bigoplus_{|n|> n_0}\range(P^\p(\iu/n))e_n;  &  \\
F^\h       &{}={}& \bigoplus_{|n|> n_0}\range(P^\h(\iu/n))e_n. &  \end{array}\]

By  \cite[Proposition 18]{BKL20}, we notably have
\[L^2(\T)^{d} {}={} F^0  {}\oplus{}  F^\p  {}\oplus{}  F^\h.\]
The space $F^\p$ is stable by the semigroup $\eu^{-t\mathcal L}$ (see the definition of $P^\p$~\cite[Proposition 5]{BKL20} and the definition of $F^\p$~\cite[Proposition 18]{BKL20}). We denote by $\mathcal L^\p$ the restriction of $\mathcal L$ to $F^\p$.

Similarly, the space $F^\h$ is stable by the semigroup $\eu^{-t\mathcal L}$. We denote by $\mathcal L^\h$ the restriction of $\mathcal L$ to $F^\h$, and $-\mathcal L^\h$ generates a strongly continuous \emph{group} of bounded operators on $F^\h$~\cite[Proposition 19]{BKL20}. 

Let $\Pi^{0}$, $\Pi^{\p}$, $\Pi^{\h}$ and $\Pi$ be the projections defined by
\[\begin{array}{r@{}c@{}c@{}c@{}c@{}c@{}c@{}c}
L^2(\T)^{d} &{}={}& F^0 & {}\oplus{} & F^\p & {}\oplus{} & F^\h; & \\
\Pi^0       &{}={}& I_{F^0} & {}+{} & 0   & {}+{}      & 0;  & \\
\Pi^\p       &{}={}& 0   &  {}+{}  & I_{F^\p} & {}+{}    & 0;  &  \\
\Pi^\h       &{}={}& 0 & {}+{} & 0 & {}+{} & I_{F^\h}; & \\
\Pi\,         &{}={}& 0 & {}+{} & I_{F^\p} & {}+{} & I_{F^\h}  & {}= \Pi^\p+\Pi^\h. \end{array}\]
These projections are bounded operators on $L^2(\set T)^d$~\cite[Proposition 18]{BKL20} (and also on every $H^k(\set T)^d$, as one can readily convince by following the proof of \cite[Proposition 18]{BKL20}).

\section{Null controllability of regular initial conditions}\label{sec-cs}
\subsection{Regular controls for regular initial conditions}\label{ssec-an}
As a technical preparation for the proof of \cref{th-main}, we need some results regarding the regularity of controls, when the control matrix is $M = I_d$.
\begin{prop}\label{th-regular-control}
Assume that $T>T_*$ (as defined in \cref{eq-def-T*}) and that $M=I_d$. Let $k, \ell\in \N$. For every $f_0\in H^k(\T)^d$, there exists $u\in H^k_0((0,T)\times \omega)^{d_\h}\times H^\ell_0((0,T)\times \omega)^{d_\p}$ such that the solution of the parabolic-transport system~\eqref{Syst} with initial condition $f_0$ and control $u$ satisfies $f(T,\cdot) = 0$.
\end{prop}

We adapt the proof of the corresponding result when $k = 0$~\cite[Theorem 2]{BKL20}. First, we prove the following adaptation of~\cite[Proposition 21]{BKL20}.
\begin{prop}\label{th-hyp}
Let $T' \in (T^*,T)$ and $k\in \set N$. If $n_0$ (in the definition of $F^0$, see \cite[Eq.~(40--42)]{BKL20}) is large enough, there exists a continuous operator 
\begin{equation*}
\mathcal U^\h \colon\begin{array}[t]{@{}c@{}l}
 H^k(\T)^{d} \times H^k_0((T',T)\times\omega)^{d_\p} & \rightarrow  H^k_0((0,T')\times\omega)^{d_\h}\\
 (f_0,u_{\p}) &\mapsto u_{\h},
\end{array}
\end{equation*}
such that for every $(f_0,u_{\p})\in  H^k(\T)^{d} \times H^k_0((T',T)\times\omega)^{d_\p}$ (where $u_{\p}$ is extended by $0$ on $(0,T')$ and $u_{\h}$ is extended by $0$ on $(T',T)$),
\begin{equation*}
\Pi^{\h} S(T;f_0, (\mathcal U^\h (f_0,u_{\p}),u_{\p})) = 0.
\end{equation*}
\end{prop}
\begin{proof}
As in \cite[\S4.3.1]{BKL20}, the conclusion of \cref{th-hyp} is equivalent to the exact controllability of the system $\partial_t f + \mathcal L^\h f = \Pi^\h (u,0)$ at time $T'$. Since $-\mathcal L^\h$ generates a strongly continuous \emph{group}, the exact controllability at time $T'$ is equivalent to the null-controllability at time $T'$, which is what we are going to prove.

When $k=0$, \cite[Proposition 23]{BKL20} is the claimed result. To extend this result to $k >0$, we use a general result of Ervedoza and Zuazua concerning the regularity of controls for regular initial data in the context of groups of operators \cite[Theorem 1.4]{EZ10}. Let $\widetilde \omega$ an open subset of $\set T$ such that $\overline{\widetilde \omega} \subset \omega$ and $T^*(\widetilde \omega) < T'$. Let $\chi \in C_c^\infty(\omega)$ such that $\chi = 1$ on $\widetilde \omega$. Let $\eta\in C^\infty_0(0,T')$. Let $z_0 \in H^k(\set T)^d$ be an initial condition. %Set $A = \mathcal L^\h$, $H = F^\h$ and $B = \chi$. 
Let $Y_{T'}$ as defined by \cite[Proposition~1.3]{EZ10} and define the control as 
\[V(t)=\eta(t)\chi(x)M^*Y(t),\]
where $Y$ is the solution to 
\[\partial_t Y-  B^* \partial_{x}^2 Y - A^* \partial_x Y + K^* Y=0\]
associated to the initial condition $Y(T')=Y_{T'}$.
According to \cite[Proposition~1.3]{EZ10}, $V(t)$ is a control that steers $z_0$ to $0$ at time $T'$. According to \cite[Theorem~1.4]{EZ10}, $Y_{T'}\in H^k(\set T)^d$ (hence $V\in L^2(0,{T'};  H^k(\omega)^d)$) and $V\in H^k(0,{T'};  L^2(\set \omega)^d)$, with estimates of the form
\[
\|V\|_{L^2(0,{T'};  H^k(\omega)^d)}^2 + \|V\|_{H^k(0,{T'};  L^2(\omega)^d)}^2 \leq C_k \|z_0\|_{H^k(\set T)^d}^2.
\]

We claim that $L^2(0,T';\ H^k_0(\omega)) \cap H^k_0(0,T';\ L^2(\omega)) \subset H^k((0,T')\times \omega)$. Indeed, for every $\tau \in \R$ and $\xi \in \R$,
\[
(1+\tau^2 + \xi^2)^k \leq C_k\big((1+\tau^2)^k + (1+\xi^2)^k\big).\]
Hence, integrating in Fourier space, 
\[
\|f\|_{H^k(\R^2)}^2 \leq C_k\big(\|f\|_{L^2(\R;H^k(\R))}^2 + \|f\|_{H^k(\R;L^2(\R))}^2\big).
\]
Recall that for $\Omega \subset \R^n$ convex\footnote{More generally, satisfiying the \emph{segment condition}, see\cite[Definition~3.21 \&{} Theorem~5.29]{AF03}.}, $H^k_0(\Omega)$ is the set of functions whose extension by zero outside $\Omega$ are $H^k(\R^n)$. Hence, $L^2(0,T';\ H^k_0(\omega)) \cap H^k_0(0,T';\ L^2(\omega))\subset H^k((0,T')\times \omega)$ as claimed, so that $V \in H^k((0,{T'})\times \omega)^d$.

Since $\eta\in C^\infty(0,T')$ and $\chi\in C^\infty_0(\omega)$, we conclude that $V \in H^k_0((0,{T'})\times \omega)^d$.
\end{proof}

For the proof of \cref{th-regular-control}, we will also use:
\begin{prop}[\cite{BKL20}, proposition 22]\label{th-par}
Let $T'\in (T^*,T)$ and $k \in \set N$. If $n_0$ is large enough, there exists a continuous operator 
\begin{equation*}
\mathcal U^\p \colon \begin{array}[t]{@{}c@{}l}
  L^2(\T)^{d} \times L^2((0,T')\times\omega)^{d_\h}  & \rightarrow  C^\infty_c((T',T)\times\omega)^{d_\p} \\
 (f_0,u_{\h}) &\mapsto u_{\p},
\end{array}
\end{equation*}
(in the sense that for any $s\in \N$, $\mathcal U^\p :
  L^2(\T)^{d} \times L^2((0,T')\times\omega)^{d_\h}  \rightarrow  H^s_0(T',T)\times\omega)^{d_\p}$ is continuous for the natural topologies associated to these spaces)
such that for every $(f_0,u_{\h})\in  L^2(\T)^{d} \times L^2((0,T')\times\omega)^{d_\h}$,
\begin{equation*}
\Pi^{\p} S(T;f_0, (u_{\h},\mathcal U^\p(f_0,\,u_{\h})) = 0.
\end{equation*}
\end{prop}
% \begin{proof}
%  The case $k=0$ is given by \cite[Proposition~22]{BKL20}. For $k\in \set N^*$, apply \cite[Proposition~22]{BKL20} to $(\partial_x^k f_0, \partial_x^k u_\h)$.
% \end{proof}

We can now prove \cref{th-regular-control} by mimicking the proof of the case $k=0$ \cite[Proposition~20 \&{} \S 4.5]{BKL20}.
\begin{proof}[Proof of \cref{th-regular-control}]
 \step{Control up to final dimensional space} We claim that there exists a closed finite codimensional space $\mathcal G$ of $H^k(\set T)^d$ and a continuous operator  $\mathcal U\colon \mathcal G \to H^k_0((0,T')\times \omega)^{d_\h}\times C_c^\infty((T',T)\times \omega)^{d_\p}$ (in the sense that for any $s\in \N$, $\mathcal U\colon \mathcal G \to H^k_0((0,T')\times \omega)^{d_\h}\times H_0^s(T',T)\times \omega)^{d_\p}$ is continuous for the natural topologies associated to these spaces)
 such that for every $f_0\in \mathcal G$, $\Pi S(T,\,f_0,\,\mathcal U f_0) = 0$.
 
 The property $\Pi S(T,\, f_0,\, (u_\h,u_\p)) = 0$ holds if
 \begin{equation}
     \left\{
     \begin{aligned}
     u_\h &= \mathcal U^\h (f_0,\,u_\p) = \mathcal U^\h_1(f_0) + \mathcal U^\h_2(u_\p),\\
     u_\p &= \mathcal U^\p(f_0,\, u_\h) = \mathcal U^\p_1(f_0) + \mathcal U^\p_2(u_\h).
     \end{aligned}
     \right.
 \end{equation}
 Set $\mathcal C = \mathcal U_1^\p + \mathcal U^2_\p\mathcal U^\h_1$. Then, the previous relations hold if
 \begin{equation}\label{eq-compact-op}
     \mathcal C f_0 = (I-\mathcal U^\p_2 \mathcal U_2^\h)u_\p.
 \end{equation}
  Since $\mathcal U^\p_2$ is continuous from $H^k_0((T',T)\times \omega)^{d_\p}$ into $C_c((T',T)\times \omega)^{d_\p}$, we deduce that the operator $\mathcal C\colon H^k_0{((T',T)\times \omega)}^{d_\p} \to H^k_0{((T',T)\times \omega)}^{d_\p}$ is compact. Thus, according to Fredholm's alternative, the relation~\eqref{eq-compact-op} holds on a closed finite codimensional space $\mathcal G$.
  
  \step{Conclusion} Dealing with the finite (co)dimensional spaces $F^0$ and $\mathcal G$ is a straightforward adaptation of \cite[\S4.5]{BKL20}; more specifically, we use \cref{th-comp-uniq} proved in Appendix \ref{app-comp-uniq} with $H = V = H^k(\set T)^d$, $U_T = H^k_0((0,T)\times \omega)^{d_\h}\times H^\ell_0((0,T)\times \omega)$,  $A = -\mathcal L$, $B = \mathds 1_\omega$, $\mathcal G = \mathcal G$ and $\mathcal F = F^0$. The control up to a  finite dimensional space hypothesis is satisfied according to the previous step. The unique continuation hypothesis is satisfied because every generalised eigenvector is a finite sum of elements of the form $X_n \eu^{\iu n x}$ ($X_n\in \mathbb C^d$), and finite linear combinations of $X_n \eu^{\iu n x}$ have the unique continuation property thanks to, \textit{e.g.}, Jerison-Lebeau's spectral inequality (see \cite[Theorem~3]{LZ98}, or \cite[Eq.~(90)]{BKL20} for our specific case).
\end{proof}

For technical reasons, we will need the control to be in the form $P(\partial_x) u$, where $P(\partial_x)$ is a constant coefficients differential operator to be chosen later on.
\begin{prop}\label{th-special-control}
Assume that $T>T_*$ (as defined in \eqref{eq-def-T*}) and that $M=I_d$. Let $k,\ell \in \set N$. Let $P$ be a nonzero polynomial with complex coefficients.
Assume that $\ell\geqslant \deg(P)$.
Let $f_0\in H^k(\T)^d$ be such that for every $n\in \Z$, $P(\iu n) = 0 \implies c_n(f_0) = 0$. Then, there exists $u\in H^{k+\deg(P)}_0((0,T)\times \omega)^{d_\h} \times H^\ell_0((0,T)\times \omega)^{d_\p}$ such that the solution of the parabolic-transport system~\eqref{Syst} with initial condition $f_0$ and control $P(\partial_x)u$ satisfies $f(T,\cdot) = 0$.
\end{prop}
\begin{proof}
$k,\ell \in \set N$ with $\ell\geqslant \deg(P)$.Let $f_0 \in H^k(\set T)^d$ be such that for every $n\in \Z$, $P(\iu n) = 0 \implies c_n(f_0) = 0$. We define $\widetilde f_0 \coloneqq P(\partial_x)^{-1}f_0$ by $c_n(\widetilde f_0) \coloneqq P(\iu n)^{-1} c_n(f_0)$ if $P(\iu n) \neq 0$ and $c_n(\widetilde f_0) \coloneqq 0$ if $P(\iu n) = 0$. Note that $P(\partial_x) \widetilde f_0 = f_0$ and that $\widetilde f_0 \in H^{k+\deg(P)}_0( \omega)^{d}$. Then, applying \cref{th-regular-control} to $\widetilde f_0$ leads to the fact that there exists $\widetilde u\in H^{k+\deg(P)}_0((0,T)\times \omega)^{d_\h}\times H^\ell_0((0,T)\times \omega)^{d_\p}$ such that the solution $\widetilde f$ of the parabolic-transport system~\eqref{Syst} with initial condition $\widetilde f_0$ and control $\widetilde u$ satisfies $\widetilde f(T,\cdot) = 0$. Moreover, since $\widetilde f_0 \in H^{k+\deg(P)}_0( \omega)^{d}$ and $\widetilde u\in H^{k+\deg(P)}_0((0,T)\times \omega)^{d_\h}\times H^\ell_0((0,T)\times \omega)^{d_\p}$ with $\ell\geqslant \deg(P)$, we notably have $\widetilde f\in L^2((0,T);H^{k+\deg(P)}(\mathbb T))$. 
Hence, setting $f=P(\partial_x)\widetilde f$ and $u=P(\partial_x)\widetilde f$, and using that  $P(\partial_x)$ has constant coefficients (so that it commutes with the operator $\partial_t -  B \partial_{x}^2  + A \partial_x  + K Id)$) ensures that $f$ verifies \eqref{Syst} with initial condition $f_0$ and control $P(\partial_x)u$. Moreover, since $\widetilde f(T,\cdot)=0$, we also have$f(T,\cdot)=P(\partial_x)\widetilde f(T,\cdot)=0$, which leads to the desired result.
\end{proof}

\subsection{Algebraic solvability}\label{ssec-alg}
For $k\in \N$, we define
\begin{equation}
    \label{def_kal_k}
    \kal[k]{B_n}{M} \coloneqq \begin{pmatrix}M&B_nM&\dotsc & B_n^{k-1}M\end{pmatrix}.
\end{equation}
We prove the following variant of \cref{th-main}.
\begin{theorem}\label{th-main-p}
Assume that the hypotheses \eqref{dd1d2}--\eqref{h:A1} hold, and that $T>T_*$. Let $k \in \N$. Assume that for all $|n|\in \N$ large enough, the Kalman rank condition $\rank(\kal[k]{B_n}{M}) = d$ holds. Define the following space of functions
\[
E\coloneqq \{f\in L^2(\T)^d\colon \forall n\in \Z,\ c_n(f) \in \range(\kal{B_n}{M})\}.
\]
Set, when it is defined,
\[
\kal[k]{B_n}{M}^+ \coloneqq \kal[k]{B_n}{M}^* \left(\kal[k]{B_n}{M}\kal[k]{B_n}{M}^*\right)^{-1}.
\]
Write $\kal[k]{B_n}{M}^+$ by blocks as
\[
\kal[k]{B_n}{M}^+ = \begin{pmatrix}
L_{n,1}^\h & L_{n,1}^\p \\
\vdots & \vdots\\
L_{n,k}^\h & L_{n,k}^\p \\
\end{pmatrix},
\]
where the $L_{n,j}^\h$ are of size $m\times d_\h$ and the $L_{n,j}^\p$ are of size $m\times d_\p$. Considering the $L_{n,j}^\h$ as rational functions of $n$, and denoting their degree by $\deg(L_{n,j}^\h)$, set
\[
p \coloneqq \max_{1\leq j \leq k} \deg(n^{j-1} L_{n,j}^\h) = \max_{1\leq j \leq k} \big(j-1+ \deg(L_{n,j}^\h)\big).
\]
Then, for every $f_0\in H^{p}(\T)^d\cap E$, there exists a control $u\in L^2([0,T]\times \omega)$ such that the solution $f$ of the parabolic-transport system~\eqref{Syst} with initial condition $f_0$ satisfies $f(T,\cdot) = 0$.
\end{theorem}

The idea of the proof is to first choose a ``fictitious'' control that acts on every components. Then, we look at the Fourier coefficients of $f$. This transforms the control system~\eqref{Syst} into a family of finite-dimensional control systems. On each of these finite-dimensional system, we perform some algebraic manipulations, called algebraic solvability, that transform the fictitious control (that acted on every component) into an ``actual'' control (that acts only on $\range(M)$). 

We begin with the algebraic solvability result we will use, which is essentially taken from~\cite[\S2.1]{LL17}.
\begin{lemma}\label{th-res-alg}
Let $k\in \set N^*$. Let $\widetilde B \in \mathcal M_d(\mathbb C)$ and $\widetilde M \in \mathcal M_{m,d}(\mathbb R)$. Let $X_0 \in \C^d$ and $w\in H^{k-1}_0(0,T;\C^{mk})$. Write $w$ by blocks as 
\[
 w = \begin{pmatrix}w_1\\\vdots\\w_k\end{pmatrix},
\]
where $w_j \in H^{k-1}_0(\T;\C^m)$, and set $u = w_1 + w_2' + \dots + w_k^{(k-1)}$. Let $X,\widetilde X \in C^0(0,T;\C^d)$ be the solutions of 
\[
X' = \widetilde BX + \kal[k]{\widetilde B}{\widetilde M}w, \qquad \widetilde X' = \widetilde BX + \widetilde Mu,\qquad X(0) = \widetilde X(0) = X_0,
\]
where
\begin{equation*}
%    \label{def_kal_k}
    \kal[k]{\widetilde B}{\widetilde M} \coloneqq \begin{pmatrix}\widetilde M&\widetilde BM&\dotsc & \widetilde B^{k-1}\widetilde M\end{pmatrix}.
\end{equation*}
Then $X(T) = \widetilde X(T)$.
\end{lemma}
\begin{proof}
 Consider $ \widetilde{\mathcal M}_{k}$ the operator matrix with $d+m$ rows and $km$ columns defined by blocks as
 \[
 \widetilde {\mathcal M}_k \coloneqq \begin{pmatrix}
 0& -\widetilde M & \cdots & -\sum_{j=0}^{k-2}\partial_t^j\widetilde B^{k-2-j}\widetilde M\\
 -I & - \partial_t & \cdots & -\partial_t^{k-1}
 \end{pmatrix} = \begin{pmatrix}
 \widetilde{\mathcal M}_{k,1}\\ \widetilde{\mathcal M}_{k,2}
 \end{pmatrix}.
 \]
 %\[
 %\mathcal M_k \coloneqq \begin{tikzpicture}[baseline = (m)]
 %\matrix(m)[matrix of math nodes, left delimiter = (, right delimiter = ), inner sep = 0pt, column sep = 9mu]
 %{
 %0& -M & \cdots & -\sum_{j=0}^{k-2}\partial_t^jB^{k-2-j}M\\
 %-I & - \partial_t & \cdots & -\partial_t^{k-1}\\
 %};
 %\end{tikzpicture}.
 %\]
 Set also
 \[
 \mathcal P \colon (X,W) \in H^1_0(0,T;\C^d) \times L^2(0,T;\C^m) \to \partial_t X - \widetilde BX -\widetilde MW \in L^2(0,T;\C^d).
 \]
 We claim that 
 \begin{equation}\label{eq-comp-alg}
     \mathcal P \circ \widetilde {\mathcal M}_k = \kal[k]{\widetilde B}{\widetilde M}.
 \end{equation}
 Indeed, we have
 by blocks $\mathcal P \circ \widetilde {\mathcal M}_k = \begin{pmatrix}C_0 & \cdots & C_{k-1}\end{pmatrix}$ with
 \begin{align*}
     C_\ell &= -(\partial_t - \widetilde B)\sum_{j=0}^{\ell-1} \partial_t^j \widetilde B^{\ell-1-j}\widetilde M + \widetilde M \partial_t^{\ell}.
     \intertext{Then, remarking that this is a telescoping sum,}
     C_\ell &= -\sum_{j=1}^{\ell} \partial_t^{j} \widetilde B^{\ell-j}\widetilde M + \sum_{j=0}^{\ell-1} \partial_t^j \widetilde B^{\ell-j}\widetilde  M + \widetilde M\partial_t^\ell\\
     &= -\partial_t^\ell \widetilde M + \widetilde B^\ell \widetilde M - \widetilde M \partial_t^\ell,
 \end{align*}
 which proves the claimed formula~\eqref{eq-comp-alg}.
 
 Now, plug \cref{eq-comp-alg} into the differential equation $X' = \widetilde BX + \kal[k]{\widetilde B}{\widetilde M} w$, which gives
 \[
     X' = \widetilde BX + (\partial_t -\widetilde B)\widetilde{\mathcal M}_{k,1} w - \widetilde M\widetilde{\mathcal M}_{k,2}w.
 \]
 With $Y \coloneqq X - \widetilde{\mathcal M}_{k,1}w$, and remarking that $ \widetilde{\mathcal M}_{k,2} w = -u$, this can be written as $Y' = \widetilde BY + \widetilde Mu$. Since $w \in H^{k-1}_0(0,T;\C^{mk})$, $\widetilde{\mathcal M}_{k,1}w(0) = \widetilde{\mathcal M}_{k,1}w(T) = 0$. Hence $Y(0) = X(0) = \widetilde X(0)$ and $Y(T) = X(T)$. Thus $Y$ solves the same Cauchy problem as $\widetilde X$. This proves that $Y = \widetilde X$, hence $\widetilde X(T) = Y(T) = X(T)$.
\end{proof}

We can now prove \cref{th-main-p}.
\begin{proof}[Proof of \cref{th-main-p}]
Let $f_0 \in H^p(\T)^d$. Set $X_n(t) = c_n(f(t,\cdot))$ and $u_n(t) = c_n(u(t,\cdot))$. The desired conclusion $f(T,\cdot) = 0$ reads in Fourier as:  $\forall n \in \set Z$, $X_n(T) = 0$. Moreover, $X_n$ satisfies
 \begin{equation}\label{eq-fourier}
    \left\{\begin{array}{ll}
    X_n'(t) = B_n X_n(t) + M u_n(t),\quad& t\in(0,T),\\
    X_n(0) = c_n(f_0).
    \end{array}\right.
 \end{equation}
 
 First, let us give the idea of the proof: if $v$ steers $f_0$ to $0$ when $M = I$, we want to define $w_n$ by $c_n(v(t,\cdot)) = \kal[k]{B_n}{M} w_n$ (this is possible for $n$ large enough) and choose $u_n \coloneqq w_{n1} + w_{n2}' + \cdots + w_{nk}^{(k-1)}$. Then, according to \cref{th-res-alg}, the function $u_n$ steers $X_n$ from $c_n(f_0)$ to $0$. There are two problems with this crude choice of $u_n$: this construction only works for $n$ large enough, and more importantly, we have no guarantee that the support of $\sum u_n \eu^{\iu nx}$ is included in $[0,T]\times \omega$.
 
 The control strategy is to first bring frequencies less than $n_0$ to $0$ in time $\epsilon$ for some $n_0>0$ large enough to be chosen later and $\epsilon>0$ small enough so that $T> T_* + 2\epsilon$, and second use a refined version of the construction outlined above.
 
 \step{Control of a finite number of frequencies}\label{step-fd-control} Recall that $\Pi$ is the projection on frequencies larger than $n_0$ and that $E$ was defined in the statement of \cref{th-main-p}. We claim that for any $n_0 \in \N^*$, $\epsilon>0$ and $f_0 \in E$ there exists $u\in L^2(0,\epsilon;\ H^p_0(\omega))^m$ such that $(1-\Pi) S(\epsilon,f_0,Mu) = 0$.
 
 This property is equivalent to the null-controllability of the system~\eqref{Syst} projected on frequencies less or equal than $n_0$. The observability inequality associated with this problem~\cite[Theorem~2.44]{Coron07} is:
 \[
 \exists C>0,\ \forall g_0\in (1-\Pi)E,\ 
 \|\eu^{-\epsilon \mathcal L^*}g_0\|_{H^{-p}(\T)^d}^2 
 \leq C \int_0^\epsilon \|M^*\eu^{-t \mathcal L^*}g_0\|_{L^2(\omega)^m}^2 \diff t.
 \]
 Since $(1-\Pi)E$ is finite dimensional, this is equivalent to the unique continuation property
 \[
 \forall g_0\in (1-\Pi)E,\ 
 \Big(M^*\eu^{-t \mathcal L^*}g_0(x) = 0\text{ for }(t,x) \in (0,\epsilon) \times \omega \Big) \implies g_0 = 0.
 \]
 
 Let us prove this property. Let $g_0\in (1-\Pi)E$ such that $M^*\eu^{-t \mathcal L^*}g_0(x) = 0$  for $(t,x) \in (0,\epsilon) \times \omega$. Since finite sums of $\eu^{\iu nx}$ have the unique continuation property, we have for every $0<t<\epsilon$ and $|n|\leq n_0$,
 \[
 c_n(M^* \eu^{-t\mathcal L^*} g_0) = 0.
 \]
 We can rewrite this as
 \[
 M^* \eu^{-tB_n^*} c_n(g_0) = 0.
 \]
 Differentiating $\ell$ times in $t$ and evaluating at $t=0$, we get that for all $\ell\in \set N$ and $|n|\leq n_0$,
 \[
 M^* (B_n^*)^\ell c_n(g_0) = 0.
 \]
 Since we assumed that for $|n|>n_0$, $c_n(g_0) = 0$, this means that $c_n(g_0) \in \ker(\kal{B_n}{M}^*)$. But, by definition of $E$, $c_n(g_0) \in \range(\kal{B_n}{M})=\ker(\kal{B_n}{M}^*)^\perp$. Thus, $c_n(g_0) = 0$ and $g_0 = 0$. This proves the unique continuation property, and the claim.
 
 \step{Construction of $u_n$}
 We set $T' = T_* + \epsilon = T-\epsilon$.
 
 Let us write $\kal[k]{B_n}{M}^+ = Q(\iu n)/P(\iu n)$ where $Q$ is a polynomial with matrix coefficients, $P$ is a polynomial (with scalar coefficients). If we denote the adjugate matrix of a matrix $C$ by $\Adj(C)$, note that we may take
 \begin{gather*}
  Q(\iu n) = \kal[k]{B_n}{M}^* \Adj(\kal[k]{B_n}{M}\kal[k]{B_n}{M}^*);\\
  P(\iu n) = \det(\kal[k]{B_n}{M}\kal[k]{B_n}{M}^*).
 \end{gather*}
 
 Increasing $n_0$ if necessary, we may assume that for every $|n|>n_0$, $P(\iu n) \neq 0$. We first apply a control as in step~\ref{step-fd-control}: for any $f_0 \in E$, there exists $u\in L^2(0,\epsilon;\ H^p_0(\omega))^m$ such that $(1-\Pi) S(\epsilon,f_0,Mu) = 0$. Then, the resulting solution $f(\epsilon,\cdot)$ is such that $P(\iu n) = 0 \implies c_n(f(\epsilon,\cdot)) = 0$, since $P(\iu n) \not = 0$ for $|n|>n_0$ and $c_n(f(\epsilon,\cdot)) = 0$ for $|n|\leqslant n_0$. 
 
 We consider this $f(\epsilon,\cdot)$ as our new initial condition, that we denote by $f_\epsilon$, and we have to steer it to $0$ in time $T'$. Note that since $f_0 \in H^p(\T)$ and $u\in L^2(0,\epsilon;\ H^p_0(\omega))^d$, according to Duahmel's formula and the fact that  the semigroup $\eu^{-t\mathcal L}$ is strongly continuous on $H^p(\T)^d$, the state $f_\epsilon$ also belongs to $H^p(\T)^d$.
 
Let $\ell \in \N$ large enough. According to \cref{th-special-control}, there exists 
\[v\in H^{p+\deg P}_0((0,T')\times\omega)^{d_\h} \times H^\ell_0((0,T')\times\omega)^{d_\p}\]
such that $S(T',f_\epsilon,P(\partial_x)v) = 0$. Write $Q(\iu n)$ by blocks as:
\[
Q(\iu n) = \begin{pmatrix}
Q_1(\iu n) \\
\vdots\\
Q_k(\iu n)
\end{pmatrix} =\begin{pmatrix}
Q_1^\h(\iu n) & Q_1^\p(\iu n) \\
\vdots & \vdots\\
Q_k^\h(\iu n) & Q_k^\p(\iu n)
\end{pmatrix}.
\]
where the $Q_j(\iu n)$ are of size $m\times d$, the $Q_j^\h(\iu n)$ are of size $m\times d_\h$ and $Q_j^\p(\iu n)$ are of size $m\times d_\p$. Notice that the $L_{n,j}^\h$ defined in the statement of \cref{th-main-p} are $L_{n,j}^\h = Q_j^\h(\iu n)/P(\iu n)$. Set also
 \begin{equation*}
    w_{n}(t) \coloneqq Q(\iu n) c_n(v(t,\cdot)).
 \end{equation*}
 Write it by blocks as
 \begin{equation*}
     w_n(t) = \begin{pmatrix}
     w_{n,1}(t) \\ \vdots \\ w_{n,k}(t)
     \end{pmatrix}= \begin{pmatrix}
     Q_1(\iu n) c_n(v(t,\cdot)) \\ \vdots \\ Q_k(\iu n) c_n(v(t,\cdot))
     \end{pmatrix}.
 \end{equation*}
 Finally, set
 \begin{equation*}
    u_n(t) \coloneqq w_{n,1}(t) + w_{n,2}'(t) + \dots + w_{n,k}^{(k-1)}(t).
 \end{equation*}
 
 \step{Conclusion}
 Remark that for every $n\in \Z$,
 \begin{align*}
    \kal[k]{B_n}{M} w_n(t) &= \kal[k]{B_n}{M} Q(\iu n) c_n(v(t,\cdot))\\
    &= \kal[k]{B_n}{M}\kal[k]{B_n}{M}^* \Adj(\kal[k]{B_n}{M}\kal[k]{B_n}{M}^*)c_n(v(t,\cdot))\\
    &= \det(\kal[k]{B_n}{M}\kal[k]{B_n}{M}^*) c_n(v(t,\cdot))\\
    &= P(\iu n) c_n(v(t,\cdot)).
 \end{align*}
 Moreover, since $S(T',f_\epsilon,P(\partial_x)v) = 0$, the control $\widetilde v_n(t) \coloneqq P(\iu n)c_n(v(t,\cdot))$ steers $c_n(f_\epsilon)$ to $0$ for the system $X_n' = B_n X_n + \widetilde v_n$ in time $T'$. That is to say, $w_n$ steers $c_n(f_\epsilon)$ to $0$ for the system $X'_n = B_n X_n + \kal[k]{B_n}{M}w_n$ in time $T'$. Thus, according to \cref{th-res-alg}, $u_n$ steers $c_n(f_\epsilon)$ to $0$ for the system~\eqref{eq-fourier} in time $T'$.
 
 Thus, the control $u$ formally defined by $u \coloneqq \sum_{n\in \Z} u_n e_n$ is such that $S(f_\epsilon,T',Mu) = 0$. Notice that the previous sum is well-defined in $L^2(0,T';\ L^2(\set T))$.
 Remark that, if we define $u$ in the sense of distributions, 
 \[
 u = (Q_1(\partial_x)+\partial_t Q_2(\partial_x) + \dots + \partial_t^{k-1}Q_k(\partial_x)) v.
 \]
 Since $v$ is supported on $[0,T']\times \omega$, so is $u$. Consider the differential operator $\mathcal Q \coloneqq Q_1(\partial_x)+\partial_t Q_2(\partial_x) + \dots + \partial_t^{k-1}Q_k(\partial_x)$. We have $u=\mathcal Q w$.
 Write this operator by blocks as $\mathcal Q = \begin{pmatrix} \mathcal Q^\h & \mathcal Q^\p\end{pmatrix}$. In other words,
 \[
 \mathcal Q^\h \coloneqq Q_1^\h(\partial_x)+\partial_t Q_2^\h(\partial_x) + \dots + \partial_t^{k-1}Q_k^\h(\partial_x).
 \]
 The order of the differential operator $\mathcal Q^\h$ is at most
 \[
 \order(\mathcal Q^\h) \leq \max_{1\leq j \leq k} (j-1+ \deg(Q_j^\h)).
 \]
 Since $L_{n,j}^\h = Q^\h_j(\iu n)/P(\iu n)$, according to the definition of $p$ (see \cref{th-main-p}), $\order(\mathcal Q^\h) \leq p + \deg(P)$. Moreover, recall that $v\in H^{p+ \deg(P)}_0((0,T')\times \omega)^{d_\h} \times H^{\ell}_0((0,T')\times \omega)^{d_\p}$. Thus, if we choose $\ell \geq  \order(\mathcal Q^\p)$, $u \in L^2((0,T')\times \omega)$.
\end{proof}

\subsection{Upper bound on the loss of regularity}
\Cref{th-main-p} requires initial condition to be $H^{p}$ for some $p$. In this section, we provide an elementary upper bound on $p$.
\begin{prop}\label{th-kal-k0}
Assume that for $|n|$ large enough, the Kalman rank condition $\rank(\kal{B_n}{M}) = d$ holds. Let 
\[
k(n) \coloneqq \inf\{k\colon \rank(\kal[k]{B_n}{M}) = d\}\in \{-\infty\}\cap \mathbb N.
\]
Then, the sequence $(k(n))_{n\in \Z}$ is eventually constant when $|n|\to +\infty$.
We will denote $k_0 \coloneqq \lim_{|n|\to +\infty} k(n)$.
\end{prop}
\begin{proof}
The rank condition $\rank([B_n|M]_k) = d$ is equivalent to $\det([B_n|M]_k[B_n|M]_k^*) \neq 0$. Let $P_k(n) = \det([B_n|M]_k[B_n|M]_k^*)$. $P_k$ is a polynomial in $n$, hence if $P_k(n_0)\neq 0$ for some $n_0$, then $P_k(n) \neq 0$ for every large enough $|n|$. Thus, for every $n_0$, there exists $n_1$ such that $k(n) \leq k(n_0)$ whenever $|n|\geq n_1$. Since $k(n)$ is integer valued, it is eventually constant. 
\end{proof}

Then, we have the following version of \cref{th-main-p}.
\begin{theorem}\label{th-main-k0}
Assume that the hypotheses \eqref{dd1d2}--\eqref{h:A1} hold, that $T>T_*$ and that for all $|n|\in \N$ large enough, the Kalman rank condition $\rank(\kal{B_n}{M}) = d$ holds. Let $k_0$ as in \cref{th-kal-k0}. Let $E$ as in \cref{th-main-p}.

Then, for every $f_0\in H^{4d(k_0-1)}(\T)^d\cap E$, there exists a control $u\in L^2([0,T]\times \omega)$ such that the solution $f$ of the parabolic-transport system~\eqref{Syst} with initial condition $f_0$ satisfies $f(T,\cdot) = 0$.
\end{theorem}
The sufficient part of \cref{th-main}, as stated in the introduction is a special case of this theorem, since we always have $k_0\leqslant d$. Here is the main lemma that allows us to bound the $p$ of \cref{th-main-p} (see also \cite[Theorem 2.1]{ABDG09} for similar considerations).
\begin{lemma}
 Let $A \in \mathcal M_d(\mathbb C)_{p}[X]$ a polynomial of degree at most $p$ with $d\times d$ matrices coefficients. Assume that for some $z_0\in \C$, $A(z_0)$ is invertible. Then, $A^{-1} \in \C_{p(d-1)}^{d\times d}(X)$, \textit{i.e.}, the coefficients of $(A(z))^{-1}$ are rational functions of $z$ of degree at most $p(d-1)$.
\end{lemma}
\begin{proof}
 Write 
 \[
 A(z)^{-1} = \frac1{\det(A(z))} \Adj(A(z)),
 \]
 where $\Adj(A(z))$ is the adjugate matrix of $A(z)$. $\det(A(z))$ and $\Adj(A(z))$ are nonzero polynomials in $z$. Moreover, the coefficients of $\Adj(A(z))$ are sums of products on $d-1$ coefficients of $A(z)$. Hence, they are polynomials of degree at most $(d-1)p$.
\end{proof}

The case we are interested in is:
\begin{corollary}\label{th-pseudo-inv-deg}
With $k_0$ as in \cref{th-kal-k0}, set, when it is defined
\[
\kal[k_0]{B_n}{M}^+ \coloneqq \kal[k_0]{B_n}{M}^* \big(\kal[k_0]{B_n}{M} \kal[k_0]{B_n}{M}^*\big)^{-1}.
\]
Then, as a function of $n$, $\kal[k_0]{B_n}{M}^+ \in \C_{2(k_0-1)(2d-1)}^{d\times d}(X)$.
\end{corollary}
\begin{proof}
We have $\kal[k_0]{B_n}{M}\in \C_{2(k_0-1)}^{d\times mk_0}[X]$, hence
\[
\kal[k_0]{B_n}{M} \kal[k_0]{B_n}{M}^* \in \C_{4(k_0-1)}^{d\times d}.
\]
According to the previous lemma,
\[
\big(\kal[k_0]{B_n}{M} \kal[k_0]{B_n}{M}^*\big)^{-1} \in \C_{4(k_0-1)(d-1)}^{d\times d}(X).
\]
Hence $\kal[k_0]{B_n}{M}^+ \in \C_k^{d\times d}(X)$ with $k = 4(k_0-1)(d-1) + 2(k_0-1)=2(k_0-1)(2d-1)$.
\end{proof}

\begin{proof}[Proof of \cref{th-main-k0}]
According to \cref{th-main-p}, every initial condition in $E\cap H^{p}(\T)^d$ can be steered to $0$, where $p = \deg(\kal[k_0]{B_n}{M}^+) +k_0-1$ (degree as a rational function of $n$). But according to \cref{th-pseudo-inv-deg}, $\deg(\kal[k_0]{B_n}{M}^+)\leq 2(k_0-1)(2d-1)$. Thus $p \leq 4d(k_0-1)$.
Hence, every initial condition in $E\cap H^{4d(k_0-1)}(\T)^d$ can be steered to $0$.
\end{proof}

\section{Necessary conditions for null-controllability}\label{sec-cn}
%\subsection{The Kalman rank condition}\label{ssec-kcn}
%Since the Kalman rank condition is necessary for the control of ODE systems~\cite[Theorem~1.16]{Coron07}, and writing the parabolic-transport system in Fourier, we get the following necessary condition for the null-controllability of the parabolic-transport system.
%\begin{prop}\label{th-necessary-kalman}
%Let $T>0$ and assume that there exists $N\in\N$ and $k\in \N$ such that every initial conditions in $H^k(\T)^d\cap \{\sum_{|n|>N} X_n \eu^{\iu nx}\}$ for the parabolic-transport system~\eqref{Syst} can be steered to $0$ in time $T$. Then for every $|n|>N$, $\rank(\kal{B_n}{M}) = d$.
%\end{prop}

\newcommand{\wkb}{{\mathrm{WKB}}}
\subsection{Construction of WKB solutions}\label{sec-wkb}
We will give other necessary conditions of null-controllability using so called \emph{WKB solutions}, that we construct here. Using these kind of approximate solutions is standard for wave equation (see, e.g., \cite[pp.~426--428]{Hormander07} or \cite[Appendix~B]{Letrouit20} for a more elementary presentation) or Schrödinger equation (see, e.g.,~\cite[pp.~16--17]{Martinez02}). WKB solutions were also used to disprove observability of some $2\times 2$ parabolic-transport system with variable coefficients~\cite[\S3]{AMM22} (see also \cite[\S3]{AM22b} for a Navier-Stokes system with Maxwell's law). Our construction is a generalization of their construction for system of arbitrary size, which brings a few difficulties. For the sake of clarity, we construct WKB solutions only for systems with constant coefficients, which is enough for our purposes. But it is likely that such a construction could be adapted to a large class of variable-coefficients parabolic-transport systems of arbitrary sizes.

To disprove the observability inequality, these WKB solutions ought to be constructed for the adjoint system. But the parabolic-transport system~\eqref{Syst} and its adjoint have the same structure, so, in order to lighten the notations, we construct the WKB solutions for the system~\eqref{Syst}.

Let $\phi \in C^\infty([0,T]\times \set T;\,\set C)$ such that $\Im(\phi) \geq 0$ and $\partial_x \phi$ never vanishes. We search approximate solutions $g_h^\wkb(t,x)$ of the parabolic-transport system~\eqref{Syst} with the following ansatz, where $h>0$ is assumed to be small:
\begin{equation}
    \label{eq-wkb-ansatz}
    \left\{\begin{aligned}
    g_h^\wkb(t,x)  &= X_h(t,x) \eu^{\iu\phi(t, x)/h},\\
    X_h(t,x) & \sim \sum_{j\geq 0} h^j Y_j(t,x).
    \end{aligned}\right.
\end{equation}

We have
\begin{align*}
    \partial_x g_h^\wkb &= \left(\partial_x X_h + \frac\iu h\partial_x \phi X_h\right) \eu^{\iu\phi/h},\\
    \partial_t g_h^\wkb &= \left(\partial_t X_h + \frac\iu h\partial_t \phi X_h\right) \eu^{\iu\phi/h},\\
    \partial_x^2 g_h^\wkb &= \left(\partial_x^2 X_h +\frac{2\iu}h\partial_x \phi \partial_x X_h - \frac1{h^2}(\partial_x\phi)^2 X_h + \frac\iu h \partial_x^2 \phi X_h\right) \eu^{\iu\phi/h}.\\
\end{align*}
Assuming that this $g_h^\wkb$ is solution of the parabolic-transport system~\eqref{Syst}, we get
\begin{align*}
    0 =& \left(\partial_t - B\partial_x^2 +A \partial_x + K\right)(X_h\eu^{\iu \phi/h})\\
    =&\bigg[ \left(\partial_t -B \partial_x^2 +A\partial_x + K\right)X_h + \frac1h \left(\iu \partial_t \phi +\iu A \partial_x\phi - 
    \iu B\partial_x^2\phi -2\iu B \partial_x \phi \partial_x  \right)X_h + \frac1{h^2}  B(\partial_x \phi)^2 X_h \bigg] \eu^{\iu\phi/h}.
\end{align*}
Plugging in the asymptotic expansion of $X_h$, we get
\begin{align*}
    0 \sim& \sum_{j\geq -2}\bigg[ (\partial_x\phi)^2 B Y_{j+2}  + \left(\iu \partial_t \phi + \iu A \partial_x \phi - \iu B \partial_x^2 \phi - 2\iu B\partial_x \phi \partial_x\right)Y_{j+1} +\left(\partial_t -B\partial_x^2 + A\partial_x +K\right)Y_j\bigg]h^j,
\end{align*}
where, by convention, $Y_j = 0$ for $j< 0$.
We want to cancel each of the terms in this sum. Thus, we are looking for $(Y_j)_{j\geq 0}$ such that for all $j\geq -2$,
\begin{equation}\label{eq-wkb}
    (\partial_x\phi)^2 B Y_{j+2} + \left(\iu \partial_t \phi + \iu A \partial_x \phi - \iu B \partial_x^2 \phi - 2\iu B\partial_x \phi \partial_x\right)Y_{j+1} +\left(\partial_t -B\partial_x^2 + A\partial_x +K\right)Y_j = 0.
\end{equation}

\newcommand{\Yjp}[1][j]{Y_{#1}^\p}
\newcommand{\Yjh}[1][j]{Y_{#1}^\h}
\newcommand{\Yjhmu}[1][j]{Y_{#1,\mu}^\h}
\newcommand{\Yjhmuz}[1][j]{Y_{#1,\mu,0}^\h}
\newcommand{\Yjhnmu}[1][j]{Y_{#1,\neq\mu}^\h}
Solving this induction relation requires us to look at different projections of this equation. From now on, we will denote
\[
 Y_j = \begin{pmatrix} \Yjh\\ \Yjp \end{pmatrix}\quad \text{ with } \Yjh \in \set C^{d_\h} \text{ and } \Yjp \in \set C^{d_\p}.
\]

Then, recalling that $B = \big(\begin{smallmatrix}0&0\\0&D\end{smallmatrix}\big)$ and taking the parabolic components of \cref{eq-wkb} (\textit{i.e.}, the $d_\p$ last components), we get
\begin{equation}
    \label{eq-wkb-p}
    (\partial_x \phi)^2 D \Yjp = -\begin{pmatrix}0&I\end{pmatrix}\left[\left(\iu \partial_t \phi + \iu A \partial_x \phi - \iu B \partial_x^2 \phi - 2\iu B\partial_x \phi \partial_x\right)Y_{j-1} +\left(\partial_t -B\partial_x^2 + A\partial_x +K\right)Y_{j-2}\right].
\end{equation}
Since $D$ is invertible, this formula determines $\Yjp$ as a function of $Y_{j-1}$ and $Y_{j-2}$. 

Before looking at the other projections of \cref{eq-wkb}, let us recall that $A = \big(\begin{smallmatrix}A'&A_{12}\\A_{21}&A_{22}\end{smallmatrix}\big)$. We similarly write $K$ in blocks as $\big(\begin{smallmatrix}K'&K_{12}\\K_{21}&K_{22}\end{smallmatrix}\big)$. 
Then, taking the transport (\textit{i.e.}, the first $d_\h$) components of \cref{eq-wkb}, we get
\begin{equation}\label{eq-wkb-h-gen}
0 = (\iu \partial_t \phi+\iu \partial_x \phi A') \Yjh + \iu \partial_x \phi A_{12}\Yjp + 
\begin{pmatrix}  I & 0\end{pmatrix}(\partial_t +A \partial_x + K)Y_{j-1}.
\end{equation}

From now on, we choose $\phi$ of the form\footnote{\Cref{eq-wkb-p,eq-wkb-h-gen} with $j = 0$ implies $(\partial_t\phi + \partial_x\phi A')\Yjh[0] = 0$. If we want a non-trivial $\Yjh[0]$, this imposes $\phi$ to depend only on $x-\mu t$ for some $\mu \in \Sp(A')$.}
\begin{equation}
    \label{eq-wkb-phi}
    \phi(t,x) = \psi(x- \mu t),
\end{equation}
where $\mu$ is an eigenvalue of $A'$ an $\psi'$ never vanishes. With this $\phi$, \cref{eq-wkb-h-gen} reads
\begin{equation}\label{eq-wkb-h}
\begin{aligned}
0 &= 
\iu\psi'(x-\mu t)(A'-\mu) \Yjh + \iu \psi'(x-\mu t) A_{12}\Yjp+ 
\begin{pmatrix}  I & 0\end{pmatrix}(\partial_t +A \partial_x + K)Y_{j-1}\\
&=\iu\psi'(x-\mu t)(A'-\mu) \Yjh + \iu \psi'(x-\mu t) A_{12}\Yjp + (\partial_t +A' \partial_x + K')\Yjh[j-1] + (A_{12}\partial_x + K_{12})\Yjp[j-1].
\end{aligned}
\end{equation}

Denote by $P_{\mu}'$ the projection on the eigenspace of $A'$ associated with $\mu$ along the other eigenspaces. We consider $\Yjhmu\in \range(P'_\mu)$ defined by $\Yjhmu = P'_\mu \Yjh$. Similarly, we set $\Yjhnmu \in \ker(P'_\mu)$ as $\Yjhnmu = (I-P'_\mu) \Yjh$. Finally, we write in blocks $A'$ and $K'$ along the sum $\set R^d = \range(P'_\mu) \oplus \ker(P'_\mu)$ as
\[
A' = \begin{pmatrix} \mu &0\\0&A'_{22} \end{pmatrix}, \quad
  K' = 
\begin{pmatrix}
  K'_{11}&K'_{12}\\ 
  K'_{21}& K'_{22}
\end{pmatrix},
\]
where $A'_{22}\in\mathcal L(\ker(P'_{\mu}))$, $K'_{11} = P'_\mu K' P'_\mu \in \mathcal L(\range(P'_{\mu}))$, $K'_{12} \in \mathcal L(\ker(P'_\mu),\range(P'_\mu))$, etc. Then, projecting \cref{eq-wkb-h} on $\ker(P'_\mu)$ along $\range(P'_\mu)$ (\textit{i.e.}, multiplying by $(I-P'_\mu)$), we get
\begin{equation}
    \label{eq-wkb-nmu}
    \iu \psi'(x-\mu t) (A'_{22} - \mu) \Yjhnmu = -(I-P'_\mu) \left[\iu \psi'(x-\mu t) A_{12}\Yjp +
    \begin{pmatrix}  I & 0\end{pmatrix}(\partial_t +A \partial_x + K)Y_{j-1}\right].
\end{equation}
Since $P'_\mu$ is the projection on the eigenspace of $A'$ associated with the eigenvalue $\mu$, $A'-\mu$ is invertible on $\ker(P'_\mu)$, \textit{i.e.}, $A'_{22}-\mu$ is invertible. Hence, \cref{eq-wkb-nmu} determines $\Yjhnmu$ as a function of $\Yjp$ and $Y_{j-1}$.

Finally, we project \cref{eq-wkb-h} on $\range(P'_\mu)$, we get
\begin{equation}\label{eq-wkb-mu-1}
0 = (\partial_t + \mu \partial_x + K'_{11})\Yjhmu + K'_{12}\Yjhnmu + P'_\mu(A_{12}\partial_x + K_{12})\Yjp + \iu \psi'(x-\mu t)P'_\mu A_{12} \Yjp[j+1].
\end{equation}
We then use \cref{eq-wkb-p} to express $\Yjp[j+1]$ as
\[
\Yjp[j+1] = D_1 \Yjh + D_2 \Yjp + D_3 Y_{j-1},\ \text{ with }\ D_1 = -\frac\iu{\psi'(x-\mu t)} D^{-1}A_{21},
\]
and where $D_2$ and $D_3$ are matrix first or second-order differential operators. Their specific expressions do not matter for our purpose. Plugging this in \cref{eq-wkb-mu-1}, we get
\begin{multline}\label{eq-wkb-mu}
(\partial_t + \mu \partial_x + K'_{11} + P'_\mu A_{12}D^{-1}A_{21}P'_\mu )\Yjhmu \\
= - K'_{12}\Yjhnmu - P'_\mu(A_{12}\partial_x + K_{12})\Yjp - \iu \psi'(x-\mu t)P'_\mu A_{12}(D_1(I-P'_\mu) \Yjhnmu + D_2 \Yjp + D_3 Y_{j-1}).
\end{multline}
If we chose an initial condition $\Yjhmuz$ for $\Yjhmu$, \cref{eq-wkb-mu} determines $\Yjhmu$ as a function of $Y^\h_{j,\mu,0}$, $\Yjhnmu$, $\Yjp$ and $Y_{j-1}$. 

We have seen that if $\phi$ is given by \cref{eq-wkb-phi}, the $(Y_j)_{j\in \set N}$ that solve the WKB recurrence equation~\eqref{eq-wkb} are given by \cref{eq-wkb-p,eq-wkb-nmu,eq-wkb-mu}.

To be rigorous, we have only proved that \emph{if $(Y_j)_{j^n \set N}$ solves \cref{eq-wkb}, then $\Yjp$, $\Yjhnmu$ and $\Yjhmu$ solves \cref{eq-wkb-p,eq-wkb-nmu,eq-wkb-mu} respectively}, but not the reciprocal (which is what we are actually interested in). However, we easily rephrase the computations of this section as a sequence of equivalences:
\begin{itemize}
    \item $\forall j \geq -2$, $Y_j$ solves \cref{eq-wkb} if and only if;
    \item $\forall j \geq 0$, $\Yjp$ solves \cref{eq-wkb-p}, $\Yjhnmu$ solves \cref{eq-wkb-nmu} and $\Yjhmu$ solves \cref{eq-wkb-mu-1} if and only if; 
    \item $\forall j \geq 0$, $\Yjp$ solves \cref{eq-wkb-p}, $\Yjhnmu$ solves \cref{eq-wkb-nmu} and $\Yjhmu$ solves \cref{eq-wkb-mu}.
\end{itemize}

We summarize the computations of this section in the following proposition:
\begin{prop}
 \label{th-wkb}
 Let $\psi \in C^\infty(\set T)$ such that $\psi'$ never vanishes and $\Im(\psi)\geq 0$. Let $\mu \in \Sp(A')$ and set $\phi$ as in \cref{eq-wkb-phi}.
 
 For every $j \geq 0$, let $\Yjhmuz \in C^\infty(\set T; \ker(A'-\mu))$. Define $(\Yjp)_{j\geq -2}$, $(\Yjhnmu)_{j\geq -2}$ and $(\Yjhmu)_{j\geq -2}$ with the following recursive procedure:
 \begin{itemize}
     \item set $\Yjp[-2] = \Yjp[-1] = 0$, $\Yjhnmu[-2] = \Yjhnmu[-1] = 0$, $\Yjhmu[-2] = \Yjhmu[-1] = 0$;
     \item if $\Yjp[k]$, $\Yjhnmu[k]$, $\Yjhmu[k]$ are defined for $-2\leq k \leq j-1$, define $\Yjp$ with \cref{eq-wkb-p}, $\Yjhnmu$ with \cref{eq-wkb-nmu} and $\Yjhmu$ with \cref{eq-wkb-mu} with initial condition $\Yjhmuz$.
 \end{itemize}
 
 For $j\geq 0$, set $Y_j(t,x) = \left(\begin{smallmatrix}\Yjhmu + \Yjhnmu\\\Yjp\end{smallmatrix}\right)$. Let $q\in \set N$. Let the function $g_h^\wkb$ be defined by
\begin{equation}\label{true-WKB}
 g_h^\wkb(t,x) = \sum_{j=0}^q h^j Y_j \eu^{\iu\phi(t,x)/h}.
 \end{equation}
 Then, defining $r_h$ by
 \[
 (\partial_t - B \partial_x^2 + A\partial_x +K)g_h^\wkb(t,x) = r_h(t,x) \eu^{\iu\phi(t,x)/h},
 \]
 for every $k \in \N$, $\ell \in \N$, $t \in [0,T]$ and $x\in \T$,
 \[
 |\partial_t^k \partial_x^\ell r_h(t,x)|\leq C_{k,\ell} h^{q-1}.
 \]
\end{prop}

\begin{remark}\label{rk-wkb-phase}
Assume that $h^{-1} \in \set N$. Then, replacing $\phi$ by $\phi +2k\pi$ in \cref{eq-wkb-ansatz} does not change the WKB solution $g_h^\wkb$. Hence, $\phi$ can be defined up to a factor $2k\pi$. That way, $\phi$ can be non-periodic, as long as $\phi \mod 2\pi$ is. Thus, we can choose
\[
\phi(t,x) = 
\iu \varphi (x-\mu t) + n_0 (x-\mu t)\text{ with }\ \mu \in \Sp(A'),\ \varphi \geq 0,\ \text{ and }\ n_0 \in \N\setminus\{0\}.
\]
\end{remark}

These WKB solutions will be used to disprove observability inequalities that often feature a projection on high frequencies. To deal with these projection on high frequencies, we will use the following lemma.
\begin{lemma}\label{th-wkb-nonstationary}
Let $n\in \Z$. Under the assumptions of \cref{th-wkb}, for every $\ell \in \N$, we have uniformly in $0\leq t \leq T$, in the limit $h \to 0^+$,
\[
 (g_h^\wkb(t,\cdot),e_n)_{L^2} = \bigO(h^\ell).
\]
\end{lemma}
\begin{proof}

 The scalar product $(g_h^\wkb(t,\cdot),\eu^{\iu nx})_{L^2}$ can be written as
 \[
 (g_h^\wkb(t,\cdot),e_n)_{L^2} = \int_\T w_{t,h,n}(x) \eu^{\iu\psi(x-\mu t)/h} \diff x,
 \]
 where 
 \[
 w_{t,h,n}(x) \coloneqq \sum_{j=0}^q h^j Y_j(t,x) \eu^{-\iu n x}.
 \]
 Note that $w_{t,h,n}$ and its derivative are uniformly bounded for $0\leq t \leq T$ and $h \leq 1$.
 Consider the differential operator $L \coloneqq(\iu\psi'(x-\mu t))^{-1}\partial_x$. Here, we use the fact that $\psi'$ never vanishes. This operator is such that
 \[
 hL \eu^{\iu\psi(x-\mu t)/h} = \eu^{\iu\psi(x-\mu t)/h}.
 \]
 Thus, denoting $L^*$ the adjoint of $L$, by integration by parts,
 \[
 (g_h^\wkb(t,\cdot),e_n)_{L^2} = h^l\int_\T (\overline{L^*})^\ell (w_{t,h,n})(x) \eu^{\iu\psi(x-\mu t)/h} \diff x.
 \]
 %According to the definition of $L$, we have $(L^*)^l= h^{l}\left(-\overline{\iu\psi'(x-\mu t)^{-1}} \partial_x\right)^l$, so that  we indeed have
 The operator $L^*$ is a differential operator independent of $h$. Hence, by definition of $w_{t,h,n}
 $
 \[
 (g_h^\wkb(t,\cdot),e_n)_{L^2} = \bigO(h^\ell).\qedhere
 \]
\end{proof}

% \begin{remark}\label{rk-wkb-support}
% To deduce necessary conditions for the null-controllability, we need some information on the support of the $Y_j$. Here is what can be seen from their definition:
% \begin{itemize}
%     \item $\supp(\Yjp) \subset \supp(Y_{j-1}) \cup \supp(Y_{j-2})$;
%     \item $\supp(\Yjhnmu) \subset \supp(Y_{j-1}) \cup \supp(Y_{j-2})$;
%     \item $\supp(\Yjhmu(t,\cdot)) \subset (\supp(Y^\h_{j,\mu,0}) + \mu t) \cup \bigcup\limits_{0\leq s\leq t}\left(\mu(t-s) + \supp(Y_{j-1}(s,\cdot)) \cup \supp(Y_{j-2}(s,\cdot)) \right)$.
% \end{itemize}
% \end{remark}

\subsection{The parabolic-transport system is not null controllable in small time}\label{ssec-st}
We now prove that the time condition $T\geqslant T_*$ is necessary (remark that the equality case $T=T^*$ remains an open question). It was already proved to be necessary for the null-controllability of every $L^2$ initial conditions~\cite{BKL20}. But this proof did not exclude the null-controllability of every $H^k$ initial condition when $T<T_*$.
\begin{prop}\label{th-min-time}
Let $T>0$ and assume that there exists $N\in\N^*$ and $k\in \N$ such that every initial conditions in $H^k(\T)^d\cap \{\sum_{|n|>N} X_n \eu^{\iu n x}\}$ for the parabolic-transport system~\eqref{Syst} can be steered to $0$ in time $T$. Then $T\geq T_*$.
\end{prop}
\begin{proof}
 Let $\mu \in \Sp(A')$ with maximum modulus. By definition, $T_* = \ell(\omega)/|\mu|$. Let $T<T_*$.
 
 We aim to disprove that the observability inequality associated to the control problem of \cref{th-min-time} using the WKB solution constructed above. We claim that this observability inequality is: there exists $C>0$ such that for every $g_0\in L^2(\set T)^d$, the solution $g$ of
 \begin{equation}\label{eq-adj}
     (\partial_t - B^*\partial_x - A^*\partial_x + K^* \partial_x) g(t,x) = 0, \quad g(0,x) = g_0(x)
 \end{equation}
 satisfies
 \begin{equation}\label{eq-obs-time}
    \|\pi_N g(T,\cdot)\|_{H^{-k}(\set T)} \leq C \|M^*g\|_{L^2((0,T)\times \omega)}, 
 \end{equation}
 where $\pi_N\colon \sum_{n\in\set Z} X_n \eu^{\iu n x} \in L^2(\set T) \mapsto \sum_{|n|> N} X_n \eu^{\iu nx}$. 
 
 This is proved using a standard duality lemma, see \textit{e.g.}~\cite[Lemma~2.48]{Coron07} with $C_2 = \eu^{-t\mathcal L}\circ \pi_N^* \circ \iota_k$ and $C_1 \colon u\in L^2((0,T)\times \omega) \mapsto \int_0^T \eu^{-(T-t)\mathcal L} M u(t)\diff t$, where $\iota_k$ is the injection $H^k(\set T) \to L^2(\set T)$. Note that $\pi_N^*$ is the injection $\{\sum_{|n|> N} X_n \eu^{\iu nx}\} \to L^2(\set T)$, and that $\iota_k^*$ is a bijective isometry $H^{-k}(\set T) \to H^k(\set T)$ (\cite[Lemma~33]{BKL20}). %The details are left to the reader.
 
 Testing this observability inequality on initial conditions of the form $\partial_x^k g_0$ instead of $g_0$, we get
 \begin{equation}\label{eq-obs-Hk}
    \|\pi_N g(T,\cdot)\|_{L^2(\set T)} \leq C \|\partial_x^kM^*g\|_{L^2((0,T)\times \omega)}, 
 \end{equation}
 
 \step{Construction of the counterexample}
 Let $T < T_*$. There exists $x_0\notin \overline{\omega}$ such that $x_0-\mu t \notin \overline{\omega}$ for every $0\leq t \leq T$. Choose $\varphi \in C^\infty(\set T)$ real-valued such that $\varphi(x_0) = 0$, $\varphi''(x_0) =1$ and $\varphi(x)>0$ for every $x\neq x_0$. Then, choose $\phi(t,x) = \iu \varphi(x+\mu t) + (x+\mu t) n_0$, as we did in \cref{rk-wkb-phase} (the change from $\mu$ to $-\mu$ is because we are considering $-A^*$ instead of $A$).
 
 This choice of $\phi$ ensures that whatever the choices of the $Y_j$, the WKB solution $g_h^\wkb$ defined by~\cref{eq-wkb-ansatz} stays concentrated around $x_0 + \mu t$.
 
 Let $\Yjhmuz[0] \in C^\infty(\set T;\ \ker(A'^*+\mu))$ with $\Yjhmuz[0](x_0) \neq 0$.
 %supported in a small ball around $x_0$, so that $\supp(\Yjhmuz[0]) + \mu t$ is disjoint from $\omega$ for every $0\leq t \leq T$. 
 For $j\geq 1$, set $\Yjhmuz = 0$. Let $q> k+1$. Consider the function $g_h^\wkb$ defined by \cref{th-wkb} (where $B$, and $K$ are replaced respectively by $B^*$ and $K^*$, and where $A$ is replaced by $-A^*$).
 
 Set also $g_h(t,x)$ the solution of the adjoint system~\eqref{eq-adj} with initial condition $g_h^\wkb(t=0,\cdot)$.
 
 \step{Estimation of the difference between $g_h^\wkb$ and $g_h$}\label{step-error-wkb}
 According to \cref{th-wkb},
 \[
 (\partial_t - B^*\partial_x^2 - A^*\partial_x + K^*) g_h^\wkb = \bigO(h^{k+1}) \eu^{\iu\phi(t,x)/h},
 \]
 Hence, with $r_h \coloneqq g_h^\wkb - g_h$, we have $r_h(0,x) = 0$ and 
 \[
 (\partial_t - B^*\partial_x^2 - A^*\partial_x + K^*) r_h = \bigO(h^{k+1}) \eu^{\iu\phi(t,x)/h}.
 \]
 where the $\bigO$ has to be understood in  the $C^\infty$-topology.
  Since the parabolic-transport system is well-posed in $H^k(\T)^d$, we get that for every $j\in \set N$, uniformly in $0<t<T$, 
 \begin{equation}
     \label{eq-wkb-error-est}
     \|\partial_x^j(g_h^\wkb(t,\cdot) - g_h(t,\cdot))\|_{L^2} \leq C_jh^{k-j+1}.
 \end{equation}
 
 \step{Upper bound on the right-hand side of the observability inequality}
 According to the triangle inequality,
 \[
  \|\partial_x^kM^*g_h\|_{L^2((0,T)\times \omega)} \leq
  \|\partial_x^kM^*g_h^\wkb\|_{L^2((0,T)\times \omega)} +
  \|\partial_x^kM^*r_h\|_{L^2((0,T)\times \omega)}.
 \]
 According to step~\ref{step-error-wkb}, the second term of the right-hand side is $\bigO(h)$. For the first term of the right-hand side, we recall that $g_h^\wkb = \sum_{j=0}^q h^jY_j \eu^{\iu \psi(x-\mu t)}$, and that, thanks to our choice of $\psi$, $\eu^{\iu \psi(x+\mu t)}$ is exponentially small when $x+\mu t \neq x_0$. Therefore, since $x_0 - \mu t\notin \overline \omega$ for every $0 \leq t \leq T$, for some $c>0$,
 \[
  \|\partial_x^kM^*g_h^\wkb\|_{L^2((0,T)\times \omega)} = \bigO(\eu^{-c/h}).
 \]
 This proves that
 \begin{equation}\label{eq-wkb-upper-time}
  \|\partial_x^kM^*g_h\|_{L^2((0,T)\times \omega)} = \bigO(h).
 \end{equation}
 
 \step{Lower bound on the left-hand side of the observability inequality}
 According to \cref{th-wkb-nonstationary}, for any $\ell \geq 0$,
 \begin{equation}\label{eq-wkb-low-freq}
  \|\pi_N g_h^\wkb(T,\cdot)\|_{L^2(\set T)} = 
  \|g_h^\wkb(T,\cdot)\|_{L^2(\T)} + \bigO(h^\ell).
 \end{equation}
 Thus, using the inverse triangle inequality,
 \[
  \|\pi_N g_h(T,\cdot)\|_{L^2(\T)} \geq 
  \|\pi_N g_h^\wkb(T,\cdot)\|_{L^2(\T)} - \|\pi_N r_h(T,\cdot)\|_{L^2(\T)}.
 \]
 Using the error estimates of step \ref{step-error-wkb}, and \cref{eq-wkb-low-freq}, we get
 \begin{equation}\label{eq-wkb-lower-1}
  \|\pi_N g_h(T,\cdot)\|_{L^2(\T)} \geq 
  \|g_h^\wkb(T,\cdot)\|_{L^2(\T)}-Ch.
 \end{equation}
 Thus, we only need to find a lower-bound for $\|g_h^\wkb(T,\cdot)\|_{L^2(\T)}$. We have
 \[
 \|g_h^\wkb(T,\cdot)\|_{L^2(\T)}^2  =
 \int_\T \left|\sum_{j=0}^q h^j Y_j(t,x)\right|^2 \eu^{-2\varphi(x+\mu T)/h}\diff x = 
 \int_\T |Y_0(t,x)|^2 \eu^{-2\varphi(x+\mu T)/h}\diff x + \bigO(h).
 \]
 Recall that $\varphi(x_0) = 0$, that for $x\neq x_0$, $\varphi(x)$ is strictly positive and that $\varphi''(x_0) \neq 0$. Then, using Laplace's method (see \textit{e.g.}~\cite[\S2.2]{Murray84} and in particular~\cite[eq.~(2.34)]{Murray84}), we get
 \[
 \|g_h^\wkb(T,\cdot)\|_{L^2(\T)}^2 = c\sqrt h + O(h^{3/2})
 \]
 for some $c>0$. Plugging this into \cref{eq-wkb-lower-1}, we get that for $h$ small enough,
 \begin{equation}\label{eq-wkb-lower-time}
  \|\pi_N g_h(T,\cdot)\|_{L^2(\T)} \geq c\sqrt{h}.
 \end{equation}
 
 \step{Conclusion}
 Comparing the lower bound~\eqref{eq-wkb-lower-time} and the upper bound~\eqref{eq-wkb-upper-time} and taking $h$ small enough, we see that the observability inequality~\eqref{eq-obs-time} cannot hold if $T<T_*$, hence the parabolic-transport system~\eqref{Syst} with initial conditions in $H^k\cap \pi_N(L^2(\T))$ is not null-controllable in time $T<T_*$.
\end{proof}

\subsection{Rough initial conditions are not null-controllable}\label{ssec-pr}
% \begin{prop}
%  Let $\Cc_n$ be the controllability Grammian for the system $X' +B_n X = Mu$. Let $X\in \C^d$. Assume that there exists $N\in \N$ and that for every $\phi \in H^k(\T)$, there exists a control $u\in L^2((0,T)\times \omega)^m$ such that the solution $f$ of the parabolic-transport system~\eqref{Syst} with initial condition $X\phi$ satisfies $\Pi_N f(T,\cdot) = 0$.
%  Then, there exists $C>0$ such that for every $|n|>N$,
%  \[
%  |\Cc_n^{-1/2}\eu^{-TB_n}X|\leq C |n|^k.
%  \]
% \end{prop}
We now give necessary conditions for every $L^2$ initial condition to be steerable to $0$. To do this, we only need the first term of the WKB expansion of \cref{th-wkb}. By analyzing higher-order terms of the WKB expansion, it is likely that we could get necessary conditions for the null-controllability of every $H^k$ initial conditions. But doing this analysis in general seems hard, and we leave this for future work, or on a case-by-case basis.
In fact, we will prove the following statement, which is a refined version of \cref{th-necessary-rough}.
\begin{prop}
 Let $\mu\in \Sp(A')$, $N\in \N$ and $T>0$. Let $P'_\mu$ be the projection on the eigenspace of $A'$ associated to $\mu$. Write $K$ in blocks as $\big(\begin{smallmatrix}K'&K_{12}\\K_{21}&K_{22}\end{smallmatrix}\big)$, with $K'\in \mathcal M_{d_h}(\mathbb R)$. Set
 \[
 K_\mu^* \coloneqq (P'_\mu)^*\left((K')^* + A_{21}^* (D^*)^{-1} A_{12}^* \right)(P'_\mu)^*
 \]
 Assume that every initial condition $f_0\in L^2(\T)^d\cap \{\sum_{|n|>N} X_n \eu^{\iu nx}\}$ is steerable to $0$ in time $T$ with control in $L^2((0,T)\times \omega)$. Then, for every $\mu \in \Sp(A')$ and for every non-zero subspace $S\subset \range((P'_\mu)^*)$ that is stable by $K_\mu^*$, there exists $V_0\in S$ such that $M^*\big(\begin{smallmatrix} V_0\\0\end{smallmatrix}\big) \neq 0$.
\end{prop}

\begin{proof}
 \step{Observability inequality} Using a standard duality lemma~\cite[Lemma~2.48]{Coron07}, and as in the proof of \cref{th-min-time}, we get an observability inequality that is equivalent to the null-controllability of the system~\eqref{Syst} with initial conditions in $ L^2(\T)^d\cap \{\sum_{|n|>N} X_n \eu^{\iu nx}\}$. This observability inequality is: there exists $C>0$ such that for every $g_0\in L^2(\set T)^d$, the solution $g$ of
 \begin{equation}\label{eq-adj-rough}
     (\partial_t - B^*\partial_x^2 - A^*\partial_x + K^*) g(t,x) = 0, \quad g(0,x) = g_0(x)
 \end{equation}
 satisfies
 \begin{equation}\label{eq-obs-rough}
    \|\pi_N g(T,\cdot)\|_{L^2(\set T)} \leq C \|M^*g\|_{L^2((0,T)\times \omega)},
 \end{equation}
 where, as in the proof of \cref{th-min-time}, $\pi_N\colon \sum_{n\in\set Z} X_n \eu^{\iu n x} \in L^2(\set T) \mapsto \sum_{|n|> N} X_n \eu^{\iu nx}$.
 
 \step{Construction of the counterexample} Let $V_0 \in S\setminus\{0\}$. Set $\varphi \coloneqq = 0$ and let $\phi(t,x) = n_0(x-\mu t)$ as in \cref{rk-wkb-phase}. Set $\Yjhmuz[0](x) \coloneqq V_0$. For $j>0$, set $\Yjhmuz \coloneqq 0$. Let $g_h^\wkb$ be defined by \cref{th-wkb} with $B$ and $K$ replaced respectively by $B^*$ and $K^*$ and $A$ by $-A^*$, and with $q\geq 2$. Let $g_h$ be the solution of the parabolic-transport system~\eqref{Syst} with initial condition $g_h^\wkb(0,\cdot)$.
 
 Remark that according to \cref{th-wkb}, and in particular \cref{eq-wkb-mu},
 \[
 (\partial_t - \mu\partial_x + K_\mu^*)\Yjhmu[0] = 0. 
 \]
 Thus, $\Yjhmu[0](t,x) = \eu^{-tK_\mu^*} V_0$. In particular, since $S$ is stable by $K_\mu^*$, $\Yjhmu[0](t,x)\in S$ for all $t,x$.
 
 \step{Error estimate between $g_h^\wkb$ and $g_h$}\label{step-error-rough} Set $r_h \coloneqq g_h - g_h^\wkb$. Then $r_h(0,x) = 0$, and according to \cref{th-wkb},
 \[
 (\partial_t - B^* \partial_x^2 -A^* \partial_x + K^*) r_h = \bigO(h) \eu^{i\phi(t,x)/h}.
 \]
 Since the parabolic-transport system is well-posed in $L^2(\T)^d$, uniformly in $0\leq t\leq T$,
 \[
 \|r_h(t,\cdot)\|_{L^2(\T)} \leq Ch.
 \]
 
 \step{Upper bound of the right-hand side of the observability inequality}
 Using the error estimate of the previous step, the right-hand side of the observability inequality~\eqref{eq-obs-rough} satisfies
 \begin{align}
     \|M^*g_h\|_{L^2((0,T)\times \omega)}^2 &\leq 
     \|M^* g_h^\wkb\|_{L^2((0,T)\times \omega)}^2 + Ch\notag\\
     &\leq \|M^*\Yjh[0]\eu^{\iu \phi/h}\|_{L^2((0,T)\times \omega)}^2 + Ch\notag\\
     &= \left\|M^*\begin{pmatrix}\Yjhmu[0]\\0\end{pmatrix}\right\|_{L^2((0,T)\times \omega)}^2 {{}+Ch}\notag\\
     &= 2\pi \int_0^T \left|
       M^*\begin{pmatrix}\eu^{-tK_\mu^*}V_0\\0\end{pmatrix}
      \right|^2\diff t  + Ch,\label{eq-obs-rhs-rough}
 \end{align}
 where we used the definition of $g_h^\wkb$ for the last three inequalities.
 
 \step{Lower-bound of the left-hand side of the observability inequality}
 Using again the error estimate of step~\ref{step-error-rough}, the left-hand side of the observability inequality~\eqref{eq-obs-rough} satisfies
 \begin{align}
     \|\pi_N g_h(T,\cdot)\|_{L^2}^2 
     &\geq \|\pi_N g_h^\wkb(T,\cdot)\|_{L^2}^2 - Ch.\notag
     \intertext{Then, using the estimate on low frequencies of $g_h^\wkb$ (\cref{th-wkb-nonstationary})}
     \|\pi_N g_h(T,\cdot)\|_{L^2}^2 
     &\geq \|g_h^\wkb(T,\cdot)\|_{L^2}^2 - Ch.\notag
     \intertext{Now, using the definition of $g_h^\wkb$, and the fact that $|\eu^{\iu\phi}|=1$,}
     \|\pi_N g_h(T,\cdot)\|_{L^2}^2 
     &\geq \|\Yjhmu[0](T,\cdot)\|_{L^2}^2 - Ch.\notag\\
     &= 2\pi |\eu^{-TK_\mu^*}V_0|^2 -Ch.\label{eq-obs-lhs-rough}
 \end{align}

\step{Conclusion} Comparing the upper bound on the right-hand side of the observability inequality (\cref{eq-obs-rhs-rough}) and the lower bound on the left-hand side (\cref{eq-obs-lhs-rough}), we see that $M^*\eu^{-tK_\mu^*}V_0$ cannot vanish for every $0\leq t\leq T$. Since $\eu^{-tK_\mu^*}V_0\in S$ for every $t$, this proves the proposition.
\end{proof}

\section{Systems of two equations}\label{sec-2t2}
We apply the general theorems of the previous sections on $2\times 2$ systems. Some of these results are not new (see, \textit{e.g.},~\cite{CMRR14}). Our goal here is only to check whether our results are optimal, at least in this setting. 

\subsection{Control properties of \texorpdfstring{$2\times 2$}{2×2} systems: statements}
Here, we consider the parabolic transport-system~\eqref{Syst} with
\begin{equation}\label{eq-matrices-2d}
    B = \begin{pmatrix}
    0&0\\0&d
    \end{pmatrix}, \quad 
    A = \begin{pmatrix}
    a'&a_{12}\\a_{21}&a_{22}
    \end{pmatrix}, \quad
    K = \begin{pmatrix}
    k_{11}&k_{12}\\k_{21}&k_{22}
    \end{pmatrix},\quad 
    M = \begin{pmatrix}
    m_1\\m_2
    \end{pmatrix}.
\end{equation}
where all lower-case letters are real numbers, with $d>0$ and $a' \neq 0$. Here, we assume that $M$ has rank one. We do not need to treat the case where $\rank(M) = 2$, because it is already covered with the general theorem where there is a control on every component (see \cite[Theorem~2]{BKL20} or \cref{th-main-k0} with $k=1$): every initial condition in $L^2(\T)^d$ is null-controllable in time $T>T_*$. In the following three propositions, we detail the applications of our general theorem to eleven cases, showcasing the variety of phenomena that can appear depending on the values of every coefficients. The proofs are given in the next subsections.

\begin{prop}\label{th-2x2-h}
 Assume that $B,A,K,M$ are given by \cref{eq-matrices-2d}. Assume that $(m_1,m_2) = (1,0)$. If $(a_{21}, k_{21}) = (0,0)$, the parabolic-transport system~\eqref{Syst} is not null-controllable, whatever the time $T$ is.
 
 Let $T>\ell(\omega)/|a'|$ (where $\ell(\omega)$ is defined in \cref{eq-def-l}). 
 \begin{itemize}
     \item If $k_{21} \neq 0$, every initial condition in $L^2(\T)^2$ for the system~\eqref{Syst} can be steered to $0$ in time $T$ with $L^2$ controls.
     \item If $a_{21} \neq 0$ and $k_{21} = 0$, every initial condition $f_0 = (f_0^\h, f_0^\p)$ in $L^2(\T)^2$ such that $\int_\T f_0^\p = 0$ for the system~\eqref{Syst} can be steered to $0$ in time $T$ with $L^2$ controls.
 \end{itemize}
\end{prop}

\begin{prop}\label{th-2x2-p}
 Assume that $B,A,K,M$ are given by \cref{eq-matrices-2d}. Assume that $(m_1,m_2) = (0,1)$. If $(a_{12}, k_{12}) = (0,0)$, the parabolic-transport system~\eqref{Syst} is not null-controllable, whatever the time $T$ is.
 
 Let $T>\ell(\omega)/|a'|$.
 \begin{itemize}
     \item If $a_{12}\neq 0$ and $k_{12} \neq 0$, every initial condition in $H^1(\T)\times L^2(\T)$ for the system~\eqref{Syst} can be steered to $0$ in time $T$ with $L^2$ controls.
     \item If $a_{12}\neq 0$ and $k_{12} = 0$, every initial condition $f_0 = (f_0^\h,f_0^\p)$ in $H^1(\T)\times L^2(\T)$ such that $\int_\T f_0^\h = 0$ for the system~\eqref{Syst} can be steered to $0$ in time $T$ with $L^2$ controls.
     \item If $a_{12}= 0$ and $k_{12} \neq 0$, every initial condition in $H^2(\T)\times L^2(\T)$ for the system~\eqref{Syst} can be steered to $0$ in time $T$ with $L^2$ controls.
 \end{itemize}
 In every cases, there exists an initial condition $f_0$ in $L^2(\T)$ such that $\int_\T f_0 = 0$ that cannot be steered to $0$ in time $T$ with $L^2$ controls.
\end{prop}
In the case where $a_{21} = 0$ and $k_{21} \neq 0$, there is a gap in the regularity condition that is sufficient for the null controllability (\textit{i.e.}, $H^2\times L^2$), and the lack of null-controllability of $L^2\times L^2$ initial conditions. Are every $H^1 \times L^2$ initial conditions steerable to $0$? We conjecture that this is not the case, but \cref{th-necessary-rough} is not enough to prove so. We would need to look at the second term in the WKB expansion to find out, or use another method; maybe using a refined version of regularization properties of \cref{th-regularity}. 

We do not detail in general the case where $m_1\neq 0$ and $m_2 \neq 0$. Let us just mention that there is no regularity condition for null-controllability to hold. But depending on whether the solution of $\det([B_n,M]) = 0$ (which is a quadratic equation in $n$) are integer, there might be a condition on at most two fourier components for an initial condition to be steerable to $0$. We only detail the following case that is about the simultaneous control of a transport and a parabolic equation.

\begin{prop}\label{th-2x2-sim}
 Assume that $B,M$ are given by \cref{eq-matrices-2d}. Assume that $A = \big(\begin{smallmatrix} a'&0\\0&a_{22}\end{smallmatrix}\big)$ and $K = \big(\begin{smallmatrix} k_{11}&0\\0&k_{22}\end{smallmatrix}\big)$. Assume that $(m_1,m_2) = (1,1)$. Let $T>\ell(\omega)/|a'|$.
 \begin{itemize}
 \item If $a' \neq a_{22}$ and $k_{11} = k_{22}$, every initial condition $f_0 = (f_0^\h,f_0^\p)\in L^2(\T)^2$ such that $\int_T f_0^\h = \int_\T f_0^\p$ can be steered to zero with controls in $L^2$.
 \item If $a' \neq a_{22}$ and $k_{11} \neq k_{22}$, every initial condition in $L^2(\T)^2$ can be steered to zero with controls in $L^2$.
 \item If $a' = a_{22}$ and $\sqrt{(k_{22} - k_{11})/d} \notin \N$, every initial condition in $L^2(\T)^2$ can be steered to zero with controls in $L^2$.
 \item If $a' = a_{22}$ and $n_0 \coloneqq \sqrt{(k_{22} - k_{11})/d} \in \N$, every initial condition $f_0 = (f_0^\h,f_0^\p)\in L^2(\T)^2$ such that $c_{\pm n_0}(f_0^\h) = c_{\pm n_0}(f_0^\p)$ can be steered to zero with controls in $L^2$.
 \end{itemize}
\end{prop}
The case $a' \neq a_{22}$ and $k_{11} = k_{22}$ is not new, at least in spirit: the simultaneous controllability (equivalently, additive observability) of a heat equation and a wave equation has been studied by Zuazua~\cite[\S2.1--2.2]{Zuazua16}.

\subsection{Regularity of the free equation}
We will use some basic regularity results.
\begin{lemma}\label{th-regularity}
Let $f_0 \in H^1(\T)^{d_\h} \times L^2(\T)^{d_\p}$. For every $t>0$, $\eu^{-t\mathcal L} f_0 \in H^1(\T)^d$.

Assume in addition that $A_{12} = 0$, and that $f_0 \in H^2(\T)^{d_\h} \times L^2(\T)^{d_\p}$. For every $t>0$, $\eu^{-t\mathcal L} f_0 \in H^2(\T)^d$.
\end{lemma}

To prove it, we will use the following (sub)lemma:
\begin{lemma}\label{th-reg-parab}
Consider $\mathcal L^\p$ and $F^\p$ as defined in \cref{sec-notations} (or~\cite[\S4.1]{BKL20}). For every $t>0$, $k\in \N$ and $f_0 \in F^\p$, $\eu^{-t\mathcal L^\p}f_0 \in H^k(\T)^d$.
\end{lemma}
\begin{proof}
 Set $f(t) = \eu^{-t\mathcal L^\p} f_0$. Denote the first $d_\h$ components of $f(t)$ by $f^\h(t)$ and the last $d_\p$ components of $f(t)$ by $f^\p(t)$ (and similarly for $f_0$).
 
We will use some simple tools from \cite[\S4.4.1]{BKL20}. For the sake of readability, we redo the proof in full here.
\step{Computing $f^\h(t)$ as a function of $f^\p(t)$} Since $f(t) \in F^\p$, by definition of $F^\p$ (\cref{sec-notations}), for every $|n|>n_0$,
\[
P^\p(\iu/n) c_n(f(t)) = c_n(f(t)).
\]
Writing $P^\p(z)$ by blocks as $\big(\begin{smallmatrix} p_{11}(z)& p_{12}(z)\\p_{21}(z)&p_{22}(z)\end{smallmatrix}\big)$, and taking the first $d_\h$ components,
\[
p_{11}(\iu/n) c_n(f^\h(t)) + p_{12}(\iu/n) c_n(f^\p(t)) = c_n(f^\h(t)).
\]
Since $P^\p(0) = \big(\begin{smallmatrix} 0&0\\0&I\end{smallmatrix}\big)$, $p_{11}(0) = 0$ and for $z$ small enough, $|p_{11}(z)|<1$. Then, increasing $n_0$ if necessary, for $|n|>n_0$,
\[
c_n(f^\h(t)) = (I-p_{11}(\iu/n))^{-1} p_{12}(\iu/n) c_n(f^\p(t)).
\]
For $z\in \C$ small enough, let $G(z) = (I-p_{11}(z))^{-1} p_{12}(z)$. Then, $G$ depends holomorphically in $z$ small enough, and for $|n|>n_0$ $c_n(f^\h(t)) = G(\iu/n) c_n(f^\p(t))$.
 
\step{Conclusion} Define $\mathcal D$ the unbounded operator on $L^2(\T)^{d_\p}$ with domain $H^2(\T)^{d_\p}$ by
\[
 \mathcal D \bigg(\sum_n X_n \eu^{\iu n x}\bigg) \coloneqq 
 \sum_n \big(n^2 D -\iu n A_{22} - K_{22} -G(\iu/n)(\iu n A_{21} + K_{21}) \big) X_n \eu^{\iu nx}.
\]
Recall that
\[
 (\partial_t - D\partial_x^2 +A_{22} \partial_x + K_{22}) f^\p(t) + (A_{21}\partial_x +K_{21}) f^\h(t) = 0.
\]
 Since $c_n(f^\h(t)) = G(\iu/n) c_n(f^\p(t))$, this can be written as $(\partial_t + \mathcal D) f^\p(t) = 0$. Hence,
 \[
 f^\p(t) = \eu^{-t\mathcal D}f^\p_0 = 
 \sum_{|n|>n_0} \eu^{-t\big(n^2 D +\iu n A_{22} + K_{22} + G(\iu/n)(\iu n A_{21} + K_{21}) \big)} c_n(f_0^\p).
 \]
 Since $\Re(\Sp(D)) \subset (0,+\infty)$, $f^\p(t)$ is in every $H^k(\T)^{d_\p}$. Since the first $d_\h$ components of $f(t)$ are
 \[
 f^\h(t) = \sum_{|n|>n_0} G(\iu /n)c_n(f^\p(t))e_n,
 \]
 and since $G(\iu/n)$ is bounded as $|n| \to +\infty$, $f^\h(t)$ also belongs in every $H^k(\T)^{d_\h}$.
\end{proof}

\begin{proof}[Proof of \cref{th-regularity}] The proof consists in looking at the projection on hyperbolic (respectively parabolic) components of $\eu^{-t\mathcal L} f_0$, using the asymptotics for the hyperbolic projection. As in the previous proof, we denote the first $d_\h$ components of $f_0$ by $f^\h_0$ and the last $d_\p$ components by $f_0^\p$.

Let us also recall that according to~\cite[\S4.1]{BKL20},
\begin{equation}
\eu^{-t \mathcal L} f_0 = \eu^{-t \mathcal L^0}\Pi^0 f_0 + \eu^{-t \mathcal L^\h}\Pi^\h f_0 + \eu^{-t \mathcal L^\p}\Pi^0 f_\p.\label{eq-sum-semigroup}
\end{equation}

\step{Asymptotics for the hyperbolic projection}
 We use the notations $P^\p(z)$, $P^\h(z)$ defined in \cite[Proposition 5--6]{BKL20}. Using the series for the perturbation of the total eigenprojections~\cite[Ch.~II,~eq.~(2.14)]{Kato95}, we get
 \begin{align}
     P^\h(z) 
     &= \begin{pmatrix}
       I&0\\0&0
      \end{pmatrix} 
      -z \left(
      \begin{pmatrix}
       I&0\\0&0
      \end{pmatrix} 
      A
      \begin{pmatrix}
       0&0\\0&D^{-1}
      \end{pmatrix} +
      \begin{pmatrix}
       0&0\\0&D^{-1}
      \end{pmatrix}
      A
      \begin{pmatrix}
       I&0\\0&0
      \end{pmatrix}  \right) + \bigO(z^2)\notag\\
     &= \begin{pmatrix}
       I&0\\0&0
      \end{pmatrix} 
      -z 
      \begin{pmatrix}
       0&A_{12}D^{-1}\\D^{-1}A_{21}&0
      \end{pmatrix} + \bigO(z^2).\notag
 \end{align}
 Thus,
 \begin{equation}
      \Pi^\h f_0 
     = \sum_{|n|>n_0} \left[ \begin{pmatrix}
      c_n(f_0^\h) \\0
     \end{pmatrix} - \frac\iu n \begin{pmatrix}
      A_{12} D^{-1}c_n(f_0^\p) \\ D^{-1} A_{21} c_n(f_0^\h)
     \end{pmatrix} + \bigO(n^{-2} c_n(f_0)) \right]\eu^{\iu n x}.\label{eq-asym-h}
 \end{equation}
    
\step{Case where $f_0 \in H^1\times L^2$} 
Since $\Pi^0 f_0$ is a finite sum of $\eu^{\iu n x}$, it is in every $H^k$, and so is $\eu^{-t\mathcal L^0} \Pi^0 f_0$. According to the regularity of the parabolic frequencies (\cref{th-reg-parab}), $\eu^{-t \mathcal L^\p}\Pi^\p f_0$ is in every $H^k$.

Since $f_0^\h \in H^1(\T)^{d_\h}$, $(c_n(f_0^\h))_n \in \ell^2(\Z;\ 1+n^2)$ (the $\ell^2$ space with weight $1+n^2$). Since $f_0^\p\in L^2(\T)^{d_\p}$, $(c_n(f_0^\p))_n \in \ell^2(\Z)$. Hence, 
\[
\left(c_n(f_0^\h) - \frac\iu n A_{12} D^{-1}c_n(f_0^\p) \right)_{|n|>n_0} \in \ell^2({|n|>n_0};\ 1+n^2),
\]
and
\[
\left(D^{-1}A_{21} c_n(f_0^\h)\right)_{|n|>n_0} \in \ell^2({|n|>n_0};\ 1+n^2).
\]
Hence, according to the asymptotics for $\Pi^\h$ of \cref{eq-asym-h}, $\Pi^\h f_0 \in H^1(\T)^d$. Since $\eu^{-t\mathcal L^\h}$ is continuous on every $H^k$, $\eu^{-t\mathcal L^\h}\Pi^\h f_0 \in H^1$.

\step{Case where $f_0 \in H^2\times L^2$ and $A_{12} = 0$} The asymptotics~\eqref{eq-asym-h} reads 
 \begin{equation}
      \Pi^\h f_0 
     = \sum_{|n|>n_0} \left[ \begin{pmatrix}
      c_n(f_0^\h) \\0
     \end{pmatrix} - \frac\iu n \begin{pmatrix}
      0 \\ D^{-1} A_{21} c_n(f_0^\h)
     \end{pmatrix} + \bigO(n^{-2} c_n(f_0)) \right]\eu^{\iu n x}.\label{eq-asym-h-2}
 \end{equation}
 The rest of the proof is very similar to the previous case: $\eu^{-t\mathcal L^0}\Pi^0 f_0$ and $\eu^{-t\mathcal L^\p}\Pi^\p f_0$ are in every $H^k$, while the asymptotics~\eqref{eq-asym-h-2} proves that $\Pi^h f_0$ ``gains'' two derivatives compared to $f_0^\p$. %The details are left to the reader.
\end{proof}

\subsection{Control properties of \texorpdfstring{$2\times 2$}{2×2} systems: proofs}\label{sec-2x2-rpoofs}

\begin{proof}[Proof of \cref{th-2x2-h}]
 In this case,
 \[
 \kal{B_n}{M} = \begin{pmatrix}
  1 & \iu n a' + k_{22} \\ 0 & \iu n a_{21} + k_{21}
 \end{pmatrix}.
 \]
 In particular, $\det(\kal{B_n}{M}) = \iu n a_{21} + k_{21}$. We see that if $(a_{21},k_{21}) = (0,0)$, the Kalman rank condition never holds, whatever $n$ is. Hence, according to \cref{rk-main}, item 1, null-controllability does not hold, whatever $T$ is.
 
 Note that in our case, $\kal{B_n}{M}^+ = \kal{B_n}{M}^{-1}$ (when the right-hand side exists). Hence,
 \[
 \kal{B_n}{M}^{-1} = \frac1{\iu n a_{21} + k_{21}}\begin{pmatrix}
  \iu n a_{21} + k_{21} & -\iu n a' - k_{22} \\ 0 & 1
 \end{pmatrix}.
 \]
 In particular, with the notations of \cref{th-main-p} with $k=2$, $L_{n,1}^\h = 1$ and $L_{n,2}^\h = 0$. Thus, $p = 0$.
 
 If $k_{21} \neq 0$, $\det(\kal{B_n}{M}) = \iu n a_{21} + k_{21}$ never vanishes. In this case, $E$ (as defined in \cref{th-main-p}) is $E = L^2(\T)^2$. Hence, according to \cref{th-main-p}, every $L^2(\T)^2$ can be steered to $0$ with $L^2$ controls in time $T > \ell(\omega)/|a'|$
 
 If $a_{21}\neq 0$ and $k_{21} = 0$, the Kalman rank condition holds for every $n \neq 0$. For $n = 0$, according to the formula for $\kal{B_n}{M}$, $\rank(\kal{B_0}{M}) = \C\times\{0\}$. Thus, $E = \{(f_0^\h,f_0^\p) \in L^2(\T)^2,\ \int_\T f_0^\p = 0\}$. Therefore, according to \cref{th-main-p}, every initial condition $(f_0^\h,f_0^\p) \in L^2(\T)^2$ such that $\int_\T f_0^\p = 0$ can be steered to $0$ with controls in $L^2$ in time $T > \ell(\omega)/|a'|$.
\end{proof}

\begin{proof}[Proof of \cref{th-2x2-p}]
 In this case,
 \[
 \kal{B_n}{M} = \begin{pmatrix}
  0 & \iu n a_{12} + k_{12} \\ 1 & -n^2 d + \iu n a_{22} + k_{22}
 \end{pmatrix}.
 \]
 In particular, $\det(\kal{B_n}{M}) = -\iu n a_{12} - k_{12}$. We see that if $(a_{12},k_{12}) = (0,0)$, the Kalman rank condition never holds, whatever $n$ is. Hence, according to \cref{rk-main} item 1, null-controllability does not hold, whatever $T$ is.
 
 As in the previous proof, $\kal{B_n}{M}^+ = \kal{B_n}{M}^{-1}$. Hence,
 \[
 \kal{B_n}{M}^{-1} = \frac1{-\iu n a_{12} - k_{12}}\begin{pmatrix}
  -n^2 d + \iu n a_{22} + k_{22} & -\iu n a_{12} - k_{12} \\ -1 & 0
 \end{pmatrix}.
 \]
 In particular, with the notations of \cref{th-main-p} with $k=2$, $L_{n,1}^\h =-(-n^2 d + \iu n a_{22} + k_{22})/(\iu n a_{12} + k_{12})$ and $L_{n,2}^\h = 1/(\iu n a_{12} + k_{12})$. In particular, if $a_{12}\neq 0$, $p = \max(1,1-1) = 1$. And if $a_{12} = 0$ and $k_{12} \neq 0$, $p = \max(2,1+0) = 2$.
 
 \step{Case $a_{12}\neq 0$ and $k_{12}\neq 0$} The Kalman rank condition holds for every $n$. Hence, with the notations of \cref{th-main-p}, $p = 1$ and $E =L^2(\T)^2$, and every initial condition in $H^1(\T)^2$ can be steered to $0$ with controls in $L^2$ in time $T>\ell(\omega)/|a'|$.
 
 The strategy to control initial conditions in $H^1\times L^2$ is first to let the solution evolve freely during an arbitrarily small time, which gives a $H^1(\T)^2$ state (\cref{th-regularity}), that we can steer to $0$ according to the previous discussion.
 
 \step{Case $a_{12}\neq 0$ and $k_{12}= 0$} The case is almost the same as the previous one, except that the Kalman rank condition is not satisfied for $n=0$ (and only for $n=0$). We have $\rank(\kal{B_0}{M}) = \{0\}\times \C$ and $E = \{(f_0^\h,f_0^\p) \in L^2(\T)^2,\ \int_\T f_0^\h = 0\}$. We still have $p=1$. Hence, we can steer every initial condition $(f_0^\h,f_0^\p) \in H^1(\T)^2$ such that $\int_\T f_0^\h = 0$ an be steered to $0$ with controls in time $T>\ell(\omega)/|a'|$.
 
 As in the previous case, to control initial conditions in $H^1\times L^2$, we let the solution evolve freely, which gives a $H^1(\T)^2$ state, and preserves the property $\int_\T f_0^\h = 0$. Then, we can steer this state in time $T>\ell(\omega)/|a'|$.
 
 \step{Case $a_{12} = 0$ and $k_{12} \neq 0$} In this case, the Kalman rank condition is satisfied for every $n$, and $p = 2$. Hence, according to \cref{th-main-p}, we can steer every $H^2(\T)^2$ initial condition to $0$ in time $T>\ell(\omega)/|a'|$ with controls in $L^2$.
 
 Again, to control an initial condition in $H^2\times L^2$, we let the solution evolve freely for a small time, which gives a $H^2(\T)^2$ state (\cref{th-regularity}), that we can steer to $0$ in time $T>\ell(\omega)/|a'|$.
 
 \step{Lack of null-controllability of $L^2$ initial conditions} We have $M^* \big(\begin{smallmatrix} 1\\0\end{smallmatrix}\big) = 0$. Hence, according to \cref{th-necessary-rough}, (recall that $A'$ has size $1\times 1$), there exists a $L^2(\T)^2$ initial condition with zero average that cannot be steered to $0$.
\end{proof}

\begin{proof}[Proof of \cref{th-2x2-sim}]
 We have
 \[
 \kal{B_n}{M} = \begin{pmatrix}
  1 & \iu n a' + k_{11} \\ 1 & -dn^2 + \iu n a_{22} + k_{22}
 \end{pmatrix}.
 \]
In particular, $\det(\kal{B_n}{M}) = -dn^2 +\iu n (a_{22}-a') + k_{22} - k_{11}$. We see that for $n$ large enough, this determinant is non zero. In fact, taking the real and imaginary parts,
\begin{equation}\label{eq-kalman-sim}
\det(\kal{B_n}{M}) = 0 \eq \left\{\begin{array}{l}
-dn^2 + k_{22} - k_{11} = 0\\
n(a_{22}-a') = 0
\end{array}\right.
\end{equation}
Moreover,
\[
\kal{B_n}{M}^+ = \kal{B_n}{M}^{-1} = \frac1{\det(\kal{B_n}{M})} \begin{pmatrix}
 -dn^2 + \iu n a_{22} + k_{22}& -\iu n a' - k_{11}\\ -1&1
\end{pmatrix}.
\]
Thus,
\[
L_{n,1}^\h = \frac{-dn^2 + \bigO(n)}{-dn^2 + \bigO(n)},\quad \text{ and }\quad
L_{n,2}^\h = \frac{-1}{-dn^2 + \bigO(n)}.
\]
Thus, $p = \max(0, 1-2) = 0$.

\step{Case $a'\neq a_{22}$ and $k_{11}=k_{22}$} According to \cref{eq-kalman-sim}, the Kalman condition is satisfied for $n\neq 0$. Moreover, for $n = 0$, $\range(\kal{B_0}{M}) = \C M$, thus $E = \{(f_0^\h,f_0^\p)\in L^2(\T)^2,\ \int_\T f_0^\h = \int_\T f_0^\p\}$. The \cref{th-main-p} gives the claimed controllability result.

\step{Case $a'\neq a_{22}$ and $k_{11}\neq k_{22}$} According to \cref{eq-kalman-sim}, the Kalman condition is satisfied for every $n\in \Z$. The \cref{th-main-p} gives the claimed controllability result.

\step{Case $a'= a_{22}$ and $\sqrt{(k_{22}-k_{11})/d}\notin \N$} As in the previous case, according to \cref{eq-kalman-sim}, the Kalman condition is satisfied for every $n\in \Z$. The \cref{th-main-p} gives the claimed controllability result.

\step{Case $a'= a_{22}$ and $n_0\coloneqq\sqrt{(k_{22}-k_{11})/d}\in \N$} According to \cref{eq-kalman-sim}, the Kalman condition is satisfied for $n\neq \pm n_0$. For $n = \pm n_0$, $\range(\kal{B_{\pm n_0}}{M}) = \C M$. The \cref{th-main-p} gives the claimed controllability result.
\end{proof}

\appendix
\section{A finite dimension-uniqueness principle for the null-controllability}\label{app-comp-uniq}
In the null controllability of parabolic-transport systems, we sometimes prove null-controllability ``up to a finite dimensional space'', and then use functional analysis arguments to deal with the finite-dimensional spaces that are left~\cite{LZ98,BKL20}. In the previous articles, this was not stated as a general result. This is the purpose of this appendix.
\begin{prop}\label{th-comp-uniq}
 Let $T_0>0$. Let $H$ be a complex Hilbert space. Let $A$ be an unbounded operator on $H$ that generates a strongly continuous semigroup of bounded operator on $H$. Let $U$ be another Hilbert space and let $B\colon U\to H$ a bounded control operator. For every $T>0$, let $U_T$ be a Hilbert space that is a subspace of $L^2(0,T;\ U)$ with continuous and dense injection that satisfies the following ``extension by $0$ property'':\footnote{In the application we use here, $U = L^2(\omega)$ and $U_T = H^k_0((0,T)\times \omega)$. The hypotheses of \cref{th-comp-uniq} are tailored to allow this situation.}
 if $u \in U_T$, $a,b>0$, then the function $\widetilde u$ defined by $\widetilde u(t) = 0$ for $0<t<a$, $\widetilde u(t) = u(t-a)$ for $a<t<T+a$, and $\widetilde u(t) = 0$ for $T+a<t<T+a+b$ is in $U_{T+a+b}$.
 
 Assume that there exists a finite dimensional space $\mathcal F$ of $H$ that is stable by the semigroup $\eu^{tA}$ and a closed finite codimensional space\footnote{We do not require $\mathcal G$ to be stable by $\eu^{tA}$.} $\mathcal G$ of $H$ such that: 
 \begin{itemize}
 \item (control up to finite dimension) for every $f_0 \in \mathcal G$, there exists $u \in U_{T_0}$ such that the solution $f$ of $f' = Af +Bu$ satisfies $f(T_0) \in \mathcal F$,
 \item (unique continuation) for every $\epsilon>0$ and for every finite linear combination of generalized eigenfunctions $g_0\in H$ of $A^*$, we have ${B^*(\eu^{tA^*} g_0) = 0} \text{ on } t \in (0,\epsilon) \implies g = 0$.
 \end{itemize}
 
 Then, for every $T>T_0$ and every $f_0 \in H$, there exists $u\in U_T$ such that the solution $f$ of $f' = Af +Bu$, $f(0) = f_0$ satisfies $f(T) = 0$.
 \end{prop}
 \begin{remark}
 \begin{itemize}
\item In this proposition, we can weaken the hypothesis ``$B$ bounded'' into ``$B$ admissible'' (see~\cite[\S2.3]{Coron07}), but in this article, $B$ is always bounded.
 
 \item If the assertion ``($g_0\in H$ is a finite linear combination of generalized eigenfunctions of $A^*$ and $B^* g_0 = 0$) $\implies g_0 = 0$'' holds, the unique continuation hypothesis is satisfied by well-posedness.
 \end{itemize}
 \end{remark}
\begin{proof}
 \step{We may assume that $\mathcal F \subset \mathcal G$} We prove that if we replace $\mathcal G$ by $\mathcal F + \mathcal G$, the hypotheses are still satisfied. Let $f_0 \in \mathcal F + \mathcal G$. We write $f_0 = f_{\mathcal F} + f_{\mathcal G}$. According to the hypotheses, there exists $u\in U_{T_0}$ such that the solution $f$ of $f' = Af +Bu$, $f(0) = f_{\mathcal G}$ is such that $f(T_0) \in \mathcal F$. Then, the solution $\widetilde f$ of $\widetilde f' = A \widetilde f + Bu$, $\widetilde f(0) = f_0$ is such that
 \[
 \widetilde f(T_0) = \underbrace{\eu^{T_0A} f_{\mathcal F}}_{\in \mathcal F} + \underbrace{f(T_0)}_{\in \mathcal F}.
 \]
 
 Note that if we replace $T_0$ by any $T_1>T_0$, the hypotheses are still satisfied.
 
 \step{For $T>T_0$, the control $u \in U_T$ such that $f(T) \in \mathcal F$ may be chosen linearly and continuously in $f_0\in \mathcal G$} This is a standard proof of control theory. For $f_0 \in \mathcal G$, set
 \[
 V(f_0) \coloneqq \{u \in U_{T}\colon f(T) \in \mathcal F,\, f \text{ solves } f'=Af+Bu,\, f(0) = f_0\}.
 \]
 Since $A$ generates a strongly continuous semigroup, $V(f_0)$ is a closed affine subspace of $U_T$. Then, we can define $\mathcal U(f_0)$ as the orthogonal projection of $0$ onto $V(f_0)$ for the $U_T$-norm. Using the characterization of orthogonal projection on closed convex set, we see that $\mathcal U$ is linear. Using the fact that $A$ generates a strongly continuous semigroup, the characterization of the projection on closed convex subsets and the closed graph theorem, we see that $\mathcal U$ is bounded. %The details are left to the reader.
 
 For the rest of the proof we set $\mathcal U_T\colon \mathcal G \to U_T$ such a map. We also set
 \begin{equation}
     \mathcal N_T \coloneqq \{f_0 \in H\colon \exists u \in U_T,\, f(T) = 0,\, f\text{ solves } f' = Af +Bu,\, f(0) = f_0\}.
 \end{equation}
 
 \step{For $T\geq T_0$, $\mathcal N_T$ is a closed finite codimensional subspace of $H$}\label{step-closed}
 Set $S_0(t)$ the semigroup $\eu^{tA}$ restricted to $\mathcal F$. Since $\mathcal F$ is finite dimensional, $S_0(t)$ can be written as $\eu^{tA_0}$, where $A_0$ is a bounded operator of $\mathcal F$. Moreover, $A_0 = A_{|\mathcal F}$. In particular, $S_0$ is actually a \emph{group} of bounded operators.
 
 For $f_0 \in \mathcal G$, and $f' = Af+B\mathcal U_T f_0$, $f(0) = f_0$, we have $f(T) \in \mathcal F$, which allows us to define
 \[
 \mathcal K\colon f_0 \in \mathcal G \mapsto - S_0(-T) f(T) \in \mathcal F
 \]
 The range of this operator $\mathcal K$ satisfies $\range(\mathcal K) \subset \mathcal F$. Hence, $\mathcal K$ has finite rank and is compact. Thus, according to Fredholm's alternative, $(I+\mathcal K)\mathcal G$ is a closed subspace of $\mathcal G$ of finite codimension.
 
 Moreover, for every $f_0 \in \mathcal G$, the solution $\widetilde f$ of $\widetilde f' = A\widetilde f+ B \mathcal U_T f_0$, $\widetilde f(0) = f_0 + \mathcal K f_0$ satisfies
 \[
 \widetilde f(T) = f(T) + \eu^{TA} \mathcal K f_0 = f(T) - S_0(T)S_0(-T) f(T) = 0.
 \]
 Thus, $(I+\mathcal K)\mathcal G \subset \mathcal F_T$. According to~\cite[Proposition~11.5]{Brezis11}, this proves that $\mathcal N_T$ is closed and has finite codimension in $H$.
 
 \step{There exists $\delta >0$ such that for every $T,T'\in(T_0, T_0+\delta)$, $\mathcal N_T = \mathcal N_{T'}$} Assume $T_0<T<T'$. If $u\in \mathcal N_T$, and if we extend $u$ by $0$ on $(T,T')$, we have have $u\in \mathcal N_{T'}$. Thus $\codim(\mathcal N_{T'})\leq \codim(\mathcal N_T)$. Since $\codim(\mathcal N_T)$ is an integer, the discontinuities of $T\mapsto \codim(\mathcal N_T)$ are isolated, which proves the claim.
 
 From now on, we choose $\epsilon \in (0,\delta/2)$ arbitrarily small and we set $T_1 = T_0 +\epsilon$.
 
 \step{For $t\in (0,\epsilon)$, $(\eu^{tA^*} \mathcal N_{T_1}^\bot)^\bot \subset \mathcal N_{T_1}$}  Let $0<t<\epsilon$ and $f_0 \in (\eu^{tA^*} \mathcal N_{T_1}^\bot)^\bot$. For every $g_0 \in \mathcal N_{T_1}^\bot$, we have
 \[
 0 = \langle \eu^{tA^*} g_0, f_0\rangle = \langle g_0, \eu^{t A} f_0\rangle.
 \]
 Thus, $\eu^{tA}f_0 \in (\mathcal N_{T_1}^\bot)^\bot$. Since $\mathcal N_{T_1}$ is closed (step \ref{step-closed}), $\eu^{tA}f_0 \in \mathcal N_{T_1}$. By definition of $\mathcal N_{T_1}$ and the ``extension by $0$'' property of $U_{T_1}$, this proves that $f_0 \in \mathcal N_{T_1+t}$. According to the previous step, $\mathcal N_{T_1+t} = \mathcal N_{T_1}$, which proves the claim.
 
 \step{$\mathcal N_{T_1}^\bot$ is left-invariant by $\eu^{tA^*}$}\label{step-stable} First, consider $0<t<\epsilon$. According to the previous step, $\mathcal N_{T_1}^\bot \subset \big((\eu^{tA^*} \mathcal N_{T_1}^\bot)^\bot\big)^\bot$. Since $\mathcal N_{T_1}^\bot$ is finite dimensional hence closed, $\mathcal N_{T_1}^\bot \subset \eu^{tA^*} \mathcal N_{T_1}^\bot$. Moreover, $\dim(\eu^{tA^*} \mathcal N_{T_1}^\bot) \leq \dim (\mathcal N_{T_1}^\bot)$. Thus, for $0<t<\epsilon$, $\eu^{tA^*} \mathcal N_{T_1}^\bot = \mathcal N_{T_1}^\bot$. Thanks to the semigroup property, this is true for all $t>0$.

 \step{Unique continuation property associated to the control problem ``steer every $f_0\in H$ into $\mathcal N_{T_1}$ in time $\epsilon$ with a control in $U_\epsilon$''}\label{step-control-F}
 The control problem is, in mathematical form, the following:
 \begin{equation}\label{eq-NT-contr}
     \forall f_0 \in H,\ \exists u \in U_\epsilon,\ f(T) \in \mathcal N_{T_1},\text{ where } f' =Af + B u,\ f(0) = f_0.
 \end{equation}
 
 Let $\Pi\colon H\to H$ the orthogonal projection on $\mathcal N_{T_1}^\bot$. Set also $R_T\colon L^2(0,T; U) \to H$ the input-to-output map defined by
 \[
 R_T u \coloneqq f(T), \text{ where } f' = Af +Bu,\ f(0) = 0.
 \]
 Then, the control problem~\eqref{eq-NT-contr} is equivalent to
 \[
 \forall f_0 \in H,\ \exists u \in U_\epsilon,\ \Pi \eu^{\epsilon A} f_0 + \Pi R_{\epsilon} u = 0.
 \]
 We denote by $\iota_\epsilon$ the injection map $U_\epsilon \to L^2(0,T; U)$. Then, the previous assertion is equivalent to
 \[
 \range\big(\Pi \circ \eu^{\epsilon A}\big) \subset \range\big(\Pi \circ R_{\epsilon} \circ \iota_\epsilon \big).
 \]
 The observability inequality associated to this control problem is (see~\cite[Lemma~2.48]{Coron07}):
 \begin{equation*}
     \forall g_0 \in H,\ \|\eu^{\epsilon A^*}\circ \Pi^* g_0\| \leq C\|\iota_\epsilon^* \circ R_{\epsilon}^* \circ \Pi^* g_0\|.
 \end{equation*}
 Since $\range(\Pi^*) = \mathcal N_{T_1}^\bot$ is finite-dimensional, and since $\ker(\iota^*) = \range(\iota)^\bot = \{0\}$, this is equivalent to
 \begin{equation}\label{eq-uc-NT1-bis}
     \forall g_0 \in \mathcal N_{T_1}^\bot,\ R_{\epsilon}^* g_0 = 0 \implies \eu^{\epsilon A^*} g_0 = 0.
 \end{equation}
 To conclude, since $\mathcal N_{T_1}^\bot$ is finite dimensional and stable by $\eu^{tA^*}$, the semigroup $\eu^{tA^*}$ is in fact a \emph{group}, and in particular $\eu^{\epsilon A^*}$ is invertible on $\mathcal N_{T_1}^\bot$. Moreover, $R_{\epsilon}^* g_0(t) = B^* \eu^{(\epsilon-t)A^*} g_0$ (see~\cite[Lemma~2.47]{Coron07}). Thus, the assertion~\eqref{eq-uc-NT1-bis} is equivalent to
 \begin{equation}\label{eq-uc-NT1}
     \forall g_0 \in \mathcal N_{T_1}^\bot,\ \Big(B^* \eu^{tA^*} g_0 = 0\text{ for } 0<t<\epsilon\Big) \implies g_0 = 0.
 \end{equation}
 
 \step{Conclusion}
 The unique continuation property~\eqref{eq-uc-NT1} of the previous step is exactly the unique continuation property we assumed. Thus, according to the previous step, we can steer every $f_0 \in H$ into $\mathcal N_{T_1}$ in time $\epsilon$ with a control in $U_\epsilon$. According to the definition of $\mathcal N_{T_1}$, we can steer every $f_0 \in \mathcal N_{T_1}$ to $0$ in time $T_1 = T + \epsilon$ with a control in $U_{T_1}$. Hence, we can steer every $f_0 \in H$ to $0$ in time $T_1 + \epsilon = T + 2\epsilon$ with a control in $U_{T+2\epsilon}$. Since $\epsilon$ can be chosen arbitrarily small, this proves the proposition. 
\end{proof}

\end{document}